\documentclass[11pt]{article}
\usepackage[english]{babel}
\usepackage{amsthm, amsmath, amssymb, amsbsy, amsfonts, latexsym, euscript}
\usepackage [latin1]{inputenc}
\usepackage{graphicx}
\theoremstyle{plain}
\newtheorem{theorem}{Theorem}[section]
\newtheorem{proposition}[theorem]{Proposition}
\newtheorem{lemma}[theorem]{Lemma}
\newtheorem{corollary}[theorem]{Corollary}
\theoremstyle{definition}
\newtheorem{definition}[theorem]{Definition}
\newtheorem{example}[theorem]{Example}

\newtheorem{remark}[theorem]{Remark}

\newcommand\Div{\operatorname{Div}}
\newcommand\D{\operatorname{Deg}}
\newcommand\Set{\operatorname{Set}}

\begin{document}

\title{Semigroups, Cartier divisors and convex bodies}

\author{Askold Khovanskii\thanks{This work was partially supported by the Canadian Grant No. 156833-17.}}
\maketitle
{\it To Bernard Teissier on the occasion of his 80-th birthday}
\begin{abstract}
The theory of Newton--Okounkov bodies provides direct relations  and points out analogies between the theory of mixed volumes of convex bodies, on the one hand, and the intersection theories of Cartier divisors and  of Shokurov $b$-divisors, on the other hand. The classical inequalities between mixed volumes of convex bodies correspond to inequalities between intersection indices of nef Cartier divisors on an irreducible projective variety and between the birationally invariant intersection indices of  nef  type Shokurov $b$-divisors on an irreducible algebraic variety. Such algebraic inequalities are known as Khovanskii--Teissier inequalities.  Our proof of these inequalities is based on simple geometric inequalities on two dimensional convex bodies and on pure algebraic arguments. The classical  geometric inequalities follow from the algebraic inequalities. We collected results  from a few papers which  were  published during the last forty  five years. Some theorems  of the present paper were never stated, but all ideas needed for their proofs are contained in the published papers. So we avoid all heavy  proofs. Our goal is to present an overview of the area. Notions of  homogeneous polynomials on commutative semigroups and of their polarizations  provide an  adequate  language  for discussing  the subject. We use this language throughout  the paper.
\end{abstract}



\section{Introduction}\label{sect1}

In the paper, we discuss an interplay between convex geometry  and algebraic geometry. We  focus on analogies and direct relations between
mixed volumes of convex bodies and the  intersection index of Cartier divisors on  projective  varieties and,  more generally, the  bi-rationally invariant intersection index  of Shokurov $b$-divisors on an arbitrary algebraic variety. Throughout the paper we do not consider Weil divisors. From now on, the term {\it divisor} always means a Cartier divisor.

There are  other analogies  between convex and algebraic geometries which deal with special algebraic   varieties equipped with  actions of reductive groups, like toric and spherical varieties and spherical homogeneous spaces. These special algebraic varieties are related with special convex bodies: with integral convex polytopes  and with special classes of polytopes  like Gelfand--Zetlin polytopes, string  polytopes and so on. These analogies involve not only volumes but also combinatorics of polytopes, structure of integral  points and other invariants. We will not discuss  these relations of convex   and algebraic  geometry with only one exception: we state and sketch a proof of the BKK (Bernstein--Koushnirenko--Khovanskii) theorem which allows us to provide an algebraic proof of the classical inequalities between mixed volumes of convex bodies.

Commutative semigroups equipped with  homogeneous polynomials are basic simple objects which appear, on the one hand,  in the theory of
divisors and  Shokurov $b$-divisors and, on the other hand,  in the theory of mixed volumes of convex bodies. In section \ref{sect2},  we present  without  proof some   theorems on homogeneous polynomials on semigroups and on their polarizations, which we use in the paper.  Complete presentation of this material can be found in \cite{Kh25}.

In section \ref{sect3},  we discuss the semigroup $K(X)$ of finite dimensional spaces of rational functions on an irreducible algebraic variety $X$ and
some sub-semigroups of $K(X)$, their Grothendieck semigroups and Grothendieck groups (see  \cite{K-K10}). These semigroups  are related to the group of Shokurov $b$-divisors on $X$ and the group of  divisors on projective  models   of $X$ (the group of Shokurov $b$-divisors on $X$
is a direct limit of groups of  divisors on all projective models of $X$, see section \ref{subsec3.4},  \cite{Isk03} and  \cite{K-K14}).

In section  \ref{sect4}, we consider semigroups related to the Newton polyhedra theory. We use these semigroups in section \ref{subsec7.2}, where we
sketch  proofs of Koushnirenko's theorem and the BKK theorem (see  \cite{Kh92}).

In section \ref{sect5}, we discuss the bi-rationally invariant intersection index on the Grothendieck group $G(X)$ of the semigroup $K(X)$ of an irreducible algebraic variety~$X$. In fact, one can consider it as the intersection  index on the group of Shokurov $b$-divisors on $X$ (see section \ref{subsec3.4}).  In section \ref{sect5}, we follow our original definition of the bi-rationally invariant intersection  index   \cite{K-K10}, which is quite natural. It works over the field of complex numbers, but it is not applicable to algebraic varieties over arbitrary algebraically closed fields. One can easily prove analogous results  over arbitrary algebraically closed fields using the approach based on the Shokurov $b$-divisors. For any projective model $Y$ of $X$, the group of  divisors on $Y$ is isomorphic to a subgroup of $G(X)$. The intersection index of  divisors on $Y$ coincides with the restriction on this subgroup of the intersection index on $G(X)$. In section \ref{sect5}, we also present formulas expressing the intersection index on $G(X)$ through the function $\D(L)$ on $K(X)$  which assigns to an element $L\in K(X)$ its self-intersection  index. Such expression  is a particular case of general formulas for the polarization of a homogeneous polynomial on a semigroup.

In section \ref{sect6}, we discuss the classical results of the theory of mixed volumes of convex bodies. The present paper was inspired by these
classical results and by the BKK theorem which relates these results with algebraic geometry.

The theory of Newton--Okounkov bodies  relates the intersection theory and the  theory of convex bodies.  It is based on  study of subs-semigroups of the lattice  $\Bbb Z^n$ of integral  points in $\Bbb R^n$. These results are presented in detail in  \cite{K-K12}.

In section \ref{sect7}, we state the most
important (for our paper) theorem   from  \cite{K-K12} and apply  it to prove Koushnirenko's theorem, the BKK  theorem and to prove geometric type inequalities for the bi-rationally invariant  intersection  index on the semigroup $K(X)$.
Then we extend these inequalities for  {\it psef type} elements in the Grothendieck group $G(X)$ of $K(X)$ and for the intersection index of nef  divisors on an irreducible projective variety.

Algebraic inequalities of such type are known as Khovanskii--Teissier inequalities (first examples  of such inequalities appeared in  \cite{Tei79},  \cite{Kh88}).

In section \ref{sect8}, we prove the classical geometric inequalities on mixed volumes of convex bodies via the inequalities in algebraic geometry. In
their  turn the algebraic inequalities are based on  geometric inequalities for convex bodies in the plane which are similar to the classical isoperimetric inequality. Thus, our proof of geometric inequalities besides  pure algebraic arguments uses only simple geometric inequalities concerning planar convex bodies.

In section \ref{sect9}, we discuss a necessary and sufficient  condition on a collection of basepoint free linear systems on a projective variety, under which generic members of these linear systems have empty intersection. These conditions are similar to the Minkowski condition on a set $M$ of convex bodies in $\Bbb R^n$ under which any $n$-tuple of convex bodies in $\Bbb R^n$ containing the set $M$ has mixed volume
equal to zero (see section~\ref{sect6}).

Vladlen Timorin and the referee of this paper made many valuable suggestions that helped me improve the exposition. They also assisted with editing my English. My wife, Tatiana Belokrinitskaya, helped me in writing the article, in particular with typing and editing. I am very grateful to all of them.

\section{Polynomial functions on commutative semigroups}\label{sect2}

\subsection{Commutative semigroups and related objects}
By a {\it semigroup}, we mean a commutative semigroup with an identity element. In this section, we refer to the semigroup operation  as addition and denote the identity element  by $0$.

Two elements $a,b$ in a  semigroup $S$  are {\it equivalent} $a\sim b $  if there is  $c\in S$ such that $a+c=b+c$.  A semigroup $S$ has  the
{\it cancellation property}  if   the  relation  $a\sim b $ for $a,b\in S$  implies that $a=b$. Consider the factor-set $S/\sim$ of the semigroup
$S$ under  the equivalence relation $\sim$.

\begin{definition}
  The factor set $S/\sim$ inherits the addition operation from $S$ and
has the structure of an additive semigroup with the identity element
which is called the  {\it  Grothendieck semigroup}  $G_s (S)$ of $S$.
The natural map of $S$ to $G_s(S)$ is denoted by $\rho: S \to G_s(S)$.
\end{definition}

By construction,  the Grothendieck semigroup $G_s(S)$ has the cancellation property, and the map $\rho: S \to G_s(S)$ is universal for this property.

The {\it Grothendieck group} of  $S$ is the group $G(S)$  of formal
differences of elements from the  Grothendieck semigroup $G_s(S)$. Let
us give its formal definition.
Consider the following equivalence relation on pairs $(a,b)\in S^2$ of
elements $a,b\in S$:  pairs $(a,b)$ and $(c,d)$ are {\it $g$-equivalent
  $(a,b)\sim_{g}(c,d)$}  if $a+d\sim b+c$.
  The factor-set $S^2/\sim_g$ inherits the addition operation from the
semigroup $S^2$ and has the structure of an additive semigroup  whose zero element is the class of all pairs of the form $(a,a)$. The semigroup $G(S)=S^2/\sim_g$ is a group:
  each element $(a,b)\in G(S)$ has the inverse element $-(a,b)=(b,a)$,
i.e., $(a,b)+(b,a)=(a+b,a+b)=0$

\begin{definition} The group $G(S)$ is called the  {\it Grothendieck
group}  of the semigroup~$S$. The  natural homomorphism
of $S$ to $G(S)$ which maps   $a\in S$ to the $g$-equivalence class of
the pair $(a,0)$,
is denoted by $\pi:S\to G(S)$.
\end{definition}

The natural map $\pi$ is the composition of the natural map  $\rho :S\to
G_s(S)$ and the   embedding $\tilde {\pi}:G_s(S)\to G(S)$ which maps an
equivalence class of an element $a\in S$ to the $g$-equivalence class of
the pair $(a,0)$.  The embedding $\tilde {\pi}$ allows us to consider
$G_s(S)$ as  a sub-semigroup in  $G(S)$.  Each element in $G(S)$ can be
represented as a difference of two elements of the sub-semigroup
$G_s(S)\subset G(S)$.

\begin{definition} With a commutative  group $G$ one associates a
$\Bbb Q$-vector space $M=G\otimes_\Bbb Z \Bbb Q$. There is a natural homomorphism
of $G$ to the additive  group of the $\Bbb Q$-space $M$   which maps
$g\in G$ to the vector $g\otimes 1\in M$. Its kernel is equal to the
torsion  $TG$ of the group $G$ and it provides  an embedding of $G/TG$
to $M$. So one can consider the group $G/TG$ as a subgroup in  the
additive group of  $M$. Vectors from this subgroup span the $\Bbb
Q$-space~$M$.

If $G$ is the Grothendieck group $G(S)$ of a semigroup $S$, denote  by
$M(S)$   the  $\Bbb Q$-vector space $M=G\otimes_\Bbb Z \Bbb Q$  and call
it the $\Bbb Q$-{\it vector space associated with} $S$.
\end{definition}

One can  check the following lemma:

\begin{lemma} For any homomorphism $\tau:S\to G$, where $G$ is a
commutative group, there are unique homomorphisms $\tau_1:G_s(S)\to G$
and $\tau_2:G(S)\to G$ such that $\tau =\tau_1\circ \rho$ and
$\tau=\tau_2\circ\pi$, where $\rho:S\to G_s(S)$ and $\pi: S\to G(S)$ are
the maps defined above.

If $G$ is an additive subgroup of a $\Bbb Q$-vector space $L$, then there is a unique $\Bbb Q$-linear map $\tau_3:M(S)\to $L$  $ such that $\tau
=\tau_3\circ \tilde \pi$, where $\tilde \pi:S\to M(S)$ is the natural
map.
  \end{lemma}

\subsection{Definition of polynomial function on a semigroup}\label{subsec2.2}
In the sections \ref{subsec2.2}--\ref{subsec2.4}, we present without
proof some results on homogeneous polynomials on semigroups and on their
polarizations. A complete presentation of this material can be found in
\cite{Kh25}.

Let $\bold k$ be a field of characteristic zero.

\begin{definition} A $\bold k$-valued function  $P:S\to \bold k$  on a
semigroup $S$ is said to be a  {\it homogeneous polynomial} of degree $n$
if for any $m$-tuple $\bold a=(a_1,\dots, a_m)$  of elements  $a_i\in S$
   the function $P_{\bold a}(k_1,\dots, k_m)= P(k_1a_1+\dots+k_ma_m)$ is
a  homogeneous polynomial of degree $n$ on $m$-tuples $(k_1,\dots, k_m)$
of nonnegative integers. In the same way, one defines  a {\it
polynomial}  and   a {\it polynomial of degree $\leq n$}~on~$S$.
\end{definition}

\begin{definition} We say that $a,b\in S$ are {\it $t$-equivalent},
$a\sim_t b$, if there is  $N>0 $ such that $Na\sim Nb$.
\end{definition}

One can check the following lemma:

\begin{lemma}\label{polynomial} If $a\sim_t b$  and $P:S\to \bold k$ is
a polynomial, then $P(a)=P(b)$.
  \end{lemma}

\begin{lemma}\label{red} Any polynomial $P$ on a semigroup $S$  can be
induced  from  unique polynomials  on the Grothendieck semigroup
$G_s(S)$ and on the   Grothendieck group $G(S)$. In turn, any
polynomial on  $G(S)$ can be induced from a unique polynomial on the $\Bbb Q$-vector space  $M(S)=G(S)\otimes_\Bbb Z \Bbb Q$.
\end{lemma}

Proof of Lemma \ref{red} is based on  Lemma \ref{polynomial}.
We consider  polynomials on $G_s(S)$, $G(S)$  and on
$M(S)=G(S)\otimes_\Bbb Z \Bbb Q$  from Lemma \ref{red} as natural
extensions of the polynomial $P$ on $S$ and  denote them by the same
symbol $P$.

\subsection{Polarization of  homogeneous polynomials}\label{subsec2.3}

One can check  that {\it a function $P:S\to \bold k$ is a homogeneous
polynomial of degree 1 if and only if $P$ is a homomorphism of $S$ to
the additive group of $\bold k$}. Let $S^n$ be a direct sum of
$n$ copies a semigroup $S$.

\begin{definition} A function $MP_n:S^n\to \bold k$ is an {\it
$n$-polarization} of a function $P:S\to \bold k$ if:
\begin{enumerate}
\item  on the diagonal, it coincides with $P$, i.e., if
$a_1=\dots=a_n=a$, then
\begin{equation*}MP_n(a_1,\dots, a_n)=P(a);\end{equation*}

\item  $MP_n(a_1,\dots, a_n)$ is symmetric  with respect to permuting
the elements
$a_1,\dots,a_n$;

\item $MP_n$  is additive  in each argument.
For the first argument, it  means that for  $a_1', a_1'',
a_2,\dots,a_n\in S$, we have:
\begin{equation*}MP_n(a'_1 + a''_1, a_2,\dots, a_n)=
MP_n(a'_1,a_2,
\dots,a_n)+ MP_n(a''_1,a_2, \dots,a_n).\end{equation*}
\end{enumerate}
\end{definition}

\begin{lemma}\label{n!} If  a function $P:S\to \bold k$ admits an
$n$-polarization $MP_n:S^n\to \bold k$, then $P$ is a homogeneous degree
$n$ polynomial. Moreover, the value of $MP_n$ at  $(a_1,\dots,a_n)\in
S^n$ is equal to  the coefficient at the polynomial
$P(k_1a_1+\dots+k_na_n)$  in the monomial $k_1\cdot \ldots\cdot k_n$ divided by $n!$.
\end{lemma}

\begin{proof} Indeed, if an $n$-polarizations exists, then for any
$m$-tuple of elements $a_i\in S$ and any $m$-tuple of nonnegative
integers $k_i$ the following identity holds:
\begin{equation}\label{identity}
  P(k_1a_1+\dots+k_ma_m) =MP_n(k_1a_1+\dots+k_ma_m,
\dots,k_1a_1+\dots+k_ma_m).
\end{equation}
Additivity of $MP_n$ in each argument implies that
$P(k_1a_1+\dots+k_ma_m)$ is a homogeneous degree~$n$ polynomial in
$(k_1,\dots,k_m)$.  Identity (\ref{identity})  for $m=n$  implies
the last statement of lemma.
\end{proof}

Lemma \ref{n!} implies that if for $F:S\to \bold k$ an $n$-polarization
exists, then it is unique. The following theorem proves its existence for
any homogeneous polynomial of degree~$n$.

  \begin{theorem}
For any  degree $n$  homogeneous polynomial $P:S\to \bold k$, there
exists a unique $n$-polarization $MP_n$. Moreover, its value  at
$(a_1,\dots,a_n)\in S^n$ is given by the following formula:
\begin{equation}\label{eq1} MP_n(a_1,\dots, a_n)=\frac{1}{n!} \sum
_{1\leq k\leq n} \Bigg( \sum _{1\leq i_1\leq \dots \leq  i_k} (-1)^{n-k}
P(x+a_{i_1}+\dots+a_{i_k})\Bigg),
\end{equation}
where $x$ is an arbitrary element of $S$.
\end{theorem}

Formula (\ref{eq1} ) and Lemma \ref{polynomial}  imply that if
$a_1\sim_t b_1,\dots, a_n\sim_t b_n$, then
\begin{equation*}MP_n(a_1,\dots,a_n)=MP_n (b_1,\dots,b_n).\end{equation*}
  In particular, (\ref{eq1}) evaluates the  $n$-polarizations of the
extensions of the polynomial $P$ to the Grothendieck semigroup  $G_s(S)$
at the equivalence classes of the  points $a_1,\dots, a_n\in S$.

\subsection{Extension of  polynomials to the  Grothendieck group}\label{subsec2.4}

Using (\ref{eq1}), one can write  an explicit  formula  which
represents  the $n$-polarization of the extension of the homogeneous
polynomial  $P$ of degree $n$ on $S$ to the Grothendieck group $G(S)$ of
$S$. Let
$g_1,\dots g_n$ be
a  $n$-tuple of elements of the group $G(S)$, where $g_i$ is represented
by a pair $(a_i,b_i)$ of elements of $S$, defined up to the equivalence
$\sim_g$. In the other words, $g_i$ is represented as a formal
differences $g_i=a_i-b_i$ of elements $a_i,b_i\in S$.
To obtain such an explicit  formula, one can put  $MP_n(g_1,\dots,g_n)$ equal to
$MF_n(a_1-b_1,\dots, a_n-b_n)$ and use the  additivity of $MP_n$ in
each argument.

In Theorem \ref{maint}, we denote by  $I_n$  the set of all indices
$\{0,\dots,n\}$   from $0$ to $n$.

\begin{theorem}\label{maint} The value of the  $n$-polarization of the
extension of  $P$ to $G(S)$ at  an $n$-tuple of elements $(g_1,\dots,
g_n)\in G^n(S)$, where  $g_i=(a_i,b_i)$ (and the pair $(a_i,b_i)$ is
defined up to the equivalence $\sim_g$)  is given by the following
formula:
\begin{equation}\label{maine}
MP_n(g_1,\dots,g_n)=MP_n((a_1,b_1),\dots, (a_n,b_n))=\frac{1}{n!}
\sum_{I} (-1)^{|J|} P(\sum_{i\in I } a_i +\sum _{j\in J} b_j).
\end{equation}
Here the summation is taken over all subsets $I$ of the set $I_n$
(including the empty set $I =\emptyset$), $J=I_n\setminus I$,  and $|J|$
is the number of elements in $J$.
\end{theorem}

Theorem \ref{maint} implies an explicit formula for the extension of
homogeneous polynomials $F:S\to \bold k$ to the Grothendieck group
$G(S)$.  Each element of the group $G(S)$ is represented by a pair $(a,b)$ of elements of $S$, defined up to the equivalence $\sim_g$. In formula (\ref{maine}), however, the components $a_i,b_i$ of the pairs $(a_i,b_i)$ are elements of $S$, defined up to the equivalence $\sim_t$.

\begin{corollary}\label{ongindex} Any degree $n$ polynomial $P:S\to \Bbb
R$ has the following extension to~$G(S)$:
\begin{equation}\label{gindex}
P((a,b))=P(a-b)= \frac{1}{n!} \sum_{0\leq k\leq n} (-1)^{n-k}{n\choose
k}P( ka +(n-k)b).
\end{equation}

\end{corollary}

\begin{proof} The corollary  follows from Theorem \ref{maint} by putting
$(a_1,b_1),\dots,(a_n,b_n)$ equal to $(a,b)$.
\end{proof}


\section{Semigroups of  subspaces,  divisors, and Shokurov $b$-divisors}\label{sect3}

In this section, we discuss examples of semigroups, and of their Grothendieck semigroups and Grothendieck groups.

\subsection{Semigroup $K(X)$} Let $X$ be  an irreducible $n$-dimensional  variety over an algebraically closed field $\bold k$ and let  $\bold k (X)$ be the field  of rational functions on $X$. We define a  multiplicative   semigroup  $K(X)$ with identity element (which does not have the cancellation property).

\begin{definition}  Let $K(X)$ be the set   of all nonzero finite dimensional vector spaces of  rational functions on  $X$. The  {\it product of  spaces  $L_1, L_2\in K(X)$} is the space $L_1L_2\in K(X)$ spanned  by all functions $fg$, where $f\in L_1$, $g\in L_2$.
The set  {\it $K(X)$ with this multiplication is a commutative semigroup}. The space  $\bold 1\in K(X)$, by definition, is the space of all constant  functions on $X$.
\end{definition}

The space $\bold 1\in K(X)$ is the identity element in $K(X)$,  i.e., for any $L\in K(X)$ the identity $\bold 1 L=L$ holds. The following lemma is obvious:

\begin{lemma}\label{invertible}
\begin{enumerate}\item If $L_1,L\in K(X)$ and $L_1L=L$, then $L_1=\bold 1$;
\item if $L_1,L_2\in K(X)$, $L_1$ is invertible and $L_1\sim L_2$, then $L_1=L_2$;
 \item an element $L\in K(X)$ is invertible if and only if $\dim _\bold k L=1$.
\end{enumerate}
\end{lemma}

\begin{proof}
\begin{enumerate} \item Let $f$ be a nonzero element of $L$. Multiplication by $f$ maps $L_1$ to $L_1$. Let $g\in L_1$ be an eigenvector of that linear map with the eigenvalue $\lambda\in \bold k$. We have $fg=\lambda g$, so $f=\lambda$. Thus, the space $L$ contains only constants, i.e., $L=\bold 1$. \item If $L_1\sim L_2$, then for some $L\in K(X)$ we have $L_2L=L_1L$, or $(L_2L^{-1}_1)L=L$. Thus, from the previous statement, $L_2L_1^{-1}=\bold 1$ or $L_2=L_1$.
\item Assume  that $L$ is invertible, i.e., $LL_1=\bold 1$  for some $L_1\in K(X)$. Let $g$ be any nonzero element of $L_1$. For any $f\in L$, we have $fg=\lambda$ for some  $\lambda\in \bold k$. Thus, $f=\lambda g^{-1}$ which implies that $\dim_\Bbb c L=1$. All one-dimensional spaces in $K(X)$ are obviously invertible.
    \end{enumerate}
    \end{proof}

\subsection{Grothendieck semigroup  $G_s(X)$ and Grothendieck group $G(X)$}

In this section, we describe the Grothendieck semigroup  $G_s(X)$  and  the Grothendieck  group   $G(X)$ of the  multiplicative  semigroup $K(X)$.

\begin{definition} A function $f \in \bold k (X)$ is  {\it integral over $L\in K(X)$} if
it satisfies an equation  $f^m+a_1 f^{m-1} + \dots + a_m =0$ with a natural number $m$ and $a_i \in L^i$.
The  {\it completion $\overline{L}$ of a space $L\in K(X)$} is the set  of all  functions integral over $L$.
\end{definition}

The following theorems describe the relationship between
the completion  $L\to \overline  L$  and the equivalence relation $\sim$ in the semigroup $K(X)$. Their proofs can be found in Section 4.2 of \cite{K-K12} and in Appendix 4 of \cite{ZarS}.

\begin{theorem} For $L\in K(X)$, the  set $\overline{L}$ is a finite-dimensional   vector space, i.e., $\overline{L}\in K(X)$.
\end{theorem}

\begin{theorem} \begin{enumerate}
\item Two spaces $L_1,L_2\in K(X)$ have the same completions, i.e.,  $\overline{L}_1 = \overline{L}_2$, if and only if they are equivalent $L_1\sim L_2$;
\item any space $L\in K(X)$ is equivalents to its completion $L\sim \overline{L} $.
\end{enumerate}
\end{theorem}

\begin{corollary} Elements of the semigroup $G_s(X)$ can be identified with complete spaces $L\in K(X)$, i.e., with  spaces $L$ such that $\overline{L}=L$. The product  $L\circ M$  of  complete spaces $L,M\in G_s(S)$ is defined by the following formula: $L \circ M=\overline{LM}$.
\end{corollary}

\begin{definition} By a {\it virtual space} $\mathcal L=L/L'$, where $L_, L'\in K(X)$, we mean an element  $\mathcal L$ of the Grothendieck group $G(X)$ corresponding to the pair $(L,L')$ of elements of $K^2(X)$, defined up to the equivalence $\sim_g$.
\end{definition}

The Grothendieck semigroup $G_s(X)$ has the cancellation property and contains the identity element. Thus, it is naturally embedded in the Grothendieck group $G(X)$.

\begin{corollary}
The elements of the group $G(X)$ can be identified with the formal ratios $L_1/L_2$ of complete spaces $L_1,L_2\in G_s(X)$. Two ratios $L_1/L_2$ and $M_1/M_2$ represent the same element in $G(X)$ if and only if $L_1 \circ  M_2 = M_1\circ L_2$.
\end{corollary}

\begin{definition} With a space $L\in K(X)$, one associates a rational map $\tilde  \Phi_L:X\dashrightarrow  L^*$ of $X$ to the dual space $L^*$  which maps a point $x\in X$ to the linear function on $L$ whose value at $f\in L$ is equal to $f(x)$. The {\it Kodaira map} $\Phi_L:X\dashrightarrow  \Bbb P( L^*)$ is the composition of the map $\tilde \Phi_L$ with the natural projection of $L^*\setminus \{0\}$ to the projective space $\Bbb P(L^*)$.
\end{definition}

\subsection{Semigroups $K_\mathrm{reg}(X)$, $K_\mathrm{emb}(X)$,  and   divisors}\label{regularsemigroup}
Let $X$ be  an $n$-dimensional irreducible  projective variety.
Consider subsets $K_\mathrm{reg}(X)$ and $K_\mathrm{emb}(X)$, $K_\mathrm{emb}(X) \subset K_\mathrm{reg}(X)\subset K(X)$  in the semigroup $K(X)$,   consisting  of all spaces $L\in K(X)$ such that  the  Kodaira map $\Phi_L:X\dashrightarrow  \Bbb P(L^*)$ is a regular map and, respectively, is an embedding.
We consider the identity  $\bold 1\in K(X)$ as an element of $K_\mathrm{reg}(X)$: the space $\Bbb P(\bold 1^*)$
is a zero dimensional projective space containing one point and the map $\Phi_L$  for $L=\bold 1$ should be considered as a regular map.
The  sets $K_\mathrm{reg}(X)$, $K_\mathrm{emb}(X)$  form sub-semigroups in $K(X)$ which we denote by the same symbols $K_\mathrm{reg}(X)$, $K_\mathrm{emb}(X)$.

\begin{definition}
Denote  the Grothendieck semigroup and the Grothendieck group of the semigroup $K_\mathrm{reg}(X)$ and of the semigroup $K_\mathrm{emb}(X)$ by $G_{s,\mathrm{reg}}(X)$, $G_\mathrm{reg}(X)$, and by $G_{s,\mathrm{emb}}(X)$, $G_\mathrm{emb}(X)$ respectively.
\end{definition}

\begin{definition}
Denote the additive group  of   divisors on $X$ by $\Div(X)$. The sub-semigroup in $\Div(X)$, consisting of very ample divisors and of basepoint free divisors  we  denote by $\Div_\mathrm{emb}(X)$, and by $\Div_\mathrm{free}(X)$ correspondingly.  The  principal divisor of a rational function on $X$  we denote by $(f)$.
\end{definition}

\begin{definition}
There is a natural map $L:\Div(X)\to K(X)$  which sends a    divisor $D\in \Div(X)$  to the  finite-dimensional space $L(D)$  consisting of all rational functions $f$ on $X$ such that $(f)+D\geq 0$.
\end{definition}

One can check that a divisor $D\in \Div (X)$ is basepoint free if and only if the space $L(D)$ belongs to the semigroup $K_\mathrm{reg}(X)$.

A function $f\in L(D)$ defines  an element of   the space $(L^*)^*(D)$. If $D\in \Div_\mathrm{free}(X)$, then the map $\Phi_{L(D)}:X\dashrightarrow   \Bbb P(L^*(D))$   is regular. On the projective space $\Bbb P(L^*(D))$ a function $f\in L(D)$ defines a hyperplane  $H$ given by an equation $h=\langle f, l\rangle=0$, for $l\in L^*(D)\setminus  \{0\}$. One can check that the pull-back divisor $\Phi^*_{L(D)} (H)$ is equal to the  divisor $D+(f)$ on $X$.

\begin{definition} To a subspace $L \in K_\mathrm{reg}(X)$  there naturally corresponds
a  divisor $D(L)$ as follows: each rational function $f\in L$ defines a
hyperplane $H\subset \Bbb P (L^*)$ given by the equation  $h={\langle f,l\rangle = 0}$ for $l\in L^*\setminus  \{0\}$.  The {\it divisor $D(L)$}  is the difference $\Phi_L^*(H)-(f)$ of the
pull-back divisor $\Phi_L^*(H)$ and the principal  divisor $(f)$.
\end{definition}

One can verify that the divisor $D(L)$ is well-defined, i.e., is independent of the choice of a function $f\in L$. The map $D:K_\mathrm{reg} (X)\to \Div (X)$, which sends the space $L\in K_\mathrm{reg}(X)$ to the divisor $D(L)\in \Div_\mathrm{free}(X)$, provides a homomorphism of the semigroup $K_\mathrm{reg}(X)$ to  the additive group $\Div_\mathrm{free}(X)$ of basepoint free   divisors in $X$.

\begin{theorem}\label{regularsemigroup}  Let $X$ be  an irreducible  projective variety. Then the Grothendieck semigroup $G_{s,\mathrm{reg}}(X)$ of the semigroup $K_\mathrm{reg}(X)$  is isomorphic to  the semigroup $\Div_\mathrm{free}(X)$ of basepoint free  divisors. Moreover, the map $D:K_\mathrm{reg} \to \Div_\mathrm{free}(X)$   is an isomorphism, that is:
\begin{enumerate}
\item the map $D:K_\mathrm{reg} \to \Div_\mathrm{free}(X)$ is surjective;

\item $D(L_1)=D(L_2)$  if and only if  elements $L_1,L_2\in K_\mathrm{reg}(X)$ are equivalent $L_1\sim L_2$ in the semigroup $K_\mathrm{reg}(X)$.
    \end{enumerate}
For $D\in \Div_\mathrm{free}(X)$ the space $L(D)\in K_\mathrm{reg}$ contains each space  $L\subset K_\mathrm{reg}$,  such that $D(L)=D$.
  \end{theorem}

One can  describe the   semigroup $G_{s,\mathrm{emb}}(X)$ as follows:

\begin{theorem}\label{embesemigroup} Let $X$ be  an irreducible  projective variety. Then the Grothendieck semigroup $G_{s,\mathrm{emb}}(X)$ of the semigroup $K_\mathrm{emb}(X)$  is isomorphic to  the semigroup $\Div_\mathrm{emb}(X)$ of very ample  divisors. Moreover, the map $D:K_\mathrm{emb} \to \Div_\mathrm{emb}(X)$ is an isomorphism, that is:
\begin{enumerate}
\item the map $D:K_\mathrm{emb} \to \Div_\mathrm{emb}(X)$ is surjective;

\item $D(L_1)=D(L_2)$  if and only if  elements $L_1,L_2\in K_\mathrm{emb}(X)$ are equivalent $L_1\sim L_2$ in the semigroup $K_\mathrm{emb}(X)$.
    \end{enumerate}
 For $D\in \Div_\mathrm{emb}$ the space $L(D)\in K_\mathrm{emb}$ contains each space  $L\subset K_\mathrm{emb}$,  such that $D(L)=D$.
  \end{theorem}

One can describe the Grothendieck groups $G_\mathrm{reg}(X)$, $G_\mathrm{emb}(X)$ as follows:

\begin{theorem}\label{regulargroup} Let $X$ be  an irreducible  projective variety and let $\pi_1:G_\mathrm{reg}(X)\to G_\mathrm{emb}(X)$ and $\pi_2:G_\mathrm{emb}(X)\to G(X)$ be the homomorphisms induced by the inclusions  $K_\mathrm{reg}\subset K_\mathrm{emb}\subset K(X)$. Then $\pi_1$ is an isomorphism, and $\pi_2$ is an embedding.
Moreover, the Grothendieck groups $G_\mathrm{reg}(X)$  and $G_\mathrm{emb}(X)$ of $K_\mathrm{reg}(X)$ and $K_\mathrm{emb}(x)$ are isomorphic to   the additive group of   divisors $\Div (X)$ on $X$.
 \end{theorem}

\begin{corollary} The map $ L: \Div_\mathrm{emb}(X)\to K(X)$ of the semigroup $D_\mathrm{emb}(X)$ of very ample  divisors to the semigroup $K(X)$ can be extended to a homomorphism $\mathcal L:\Div(X)\to G(X)$ which embeds the group of  divisors on $X$ to the group~$G(X)$.
\end{corollary}

Theorems \ref{regularsemigroup}, \ref{embesemigroup}, \ref{regulargroup} are minor modifications of results found in  \cite{K-K10}.

\subsection{The Grothendieck group $G(X)$ and the group of Shokurov  $b$-divisors~on~$X$}\label{subsec3.4}
A birational isomorphism of varieties $X$ and $Y$ induces an isomorphism   of the groups $G(X)$ and $G(Y)$. A projective variety  $Y$ together with a  birational map to  $X$  is called a {\it projective model of $X$}. According to Theorem \ref{regulargroup}, the group of  divisors  $\Div(Y)$ of any projective model $Y$ of $X$ is naturally embedded to the group $G(X)$. One can show that for any finite set of elements $\{\mathcal L_i\}\subset G(X)$  there is a projective model $Y$ of $X$ such that the set  belongs to the image of the group $\Div(Y)$ (thus, the elements of the set $\{\mathcal L_i\}$
  can be considered as  divisors in $Y$).

A projective model $Y_1$ of $X$ {\it  dominates} a projective model $Y_2$ of $X$ if the birational map $\pi_{1,2}:Y_1\to Y_2$ (induced by birational isomorphisms between $Y_1$ and $X$, and between $X$ and $Y_2$)   is regular. The pull-back map $(\pi_{1,2})^*$ embeds the group $\Div (Y_2)$ to the group $\Div (Y_1)$. This embedding preserves the intersection index of divisors on $Y_2$~and~$Y_1$.

The {\it  group of Shokurov $b$-divisors on $X$} is
the direct limit of the  groups of  divisors of all  projective models of variety $X$.
By the projection formula, the group of  Shokurov $b$-divisors on an irreducible $n$-dimensional variety $X$ inherits the intersection index of any $n$-tuple of elements. The intersection theory of Shokurov $b$-divisors can be reduced to the classical intersection theory of  divisors on projective varieties. It works for algebraic varieties over an arbitrary algebraically closed field $\bold k$.

One can see  that the group $G(X)$ is naturally isomorphic to the group of Shokurov $b$-divisors on $X$. Thus, an intersection  index is defined for any $n$-tuple of elements~in~$G(X)$.

In  \cite{K-K10}, we constructed such an intersection index on the group $G(X)$ for an irreducible variety $X$ over the field $\Bbb C$ of complex numbers. We were inspired by the BKK theorem and we did not know  about Shokurov $b$-divisors. Our construction does not rely on classical intersection theory and is rather elementary. Unfortunately, it does not work over arbitrary algebraically closed fields $\bold k$. In section \ref{sect5}, we discuss our version of the bi-rationally invariant intersection  index on $G(X)$. Using Shokurov $b$-divisors, one can easily modify the results of section in such a way that they become applicable to an algebraic variety $X$ over an arbitrary algebraically closed field $\bold k$ (see~\cite{K-K14}).

\section{Semigroups  related to Newton polyhedra theory}\label{sect4}

In this section, we discuss two semigroups which are naturally related to Newton polyhedra theory.

\subsection{The Semigroup $\Set_n$}

One can define the sum of two subsets $A,B$ in a commutative group $G$ as follows: {\it the set $A+B$ is the set of all elements $z\in G$ representable as $z=a+b$ for some  $a\in A$, $b\in B$}.

\begin{definition} The {\it semigroup}  $\Set_n$ is the semigroup  of all
finite subsets in the lattice $\Bbb Z^n\subset \Bbb R^n$ with respect to the above
addition. The Grothendieck semigroup  and the Grothendieck group of $\Set_n$ are denoted by $G_s(\Set_n)$ and $G(\Set_n)$ respectively.
\end{definition}

\begin{theorem}\label{theorem4.1} The Grothendieck semigroup  $G_s(\Set_n)$ of $\Set_n$ is isomorphic to the semigroup of convex integral polytopes in $\Bbb R^n$, i.e., of convex polytopes whose  vertices belong to $\Bbb Z^n$.  The natural map $\rho:\Set_n\to G_s(\Set_n)$ sends  an element $A\in \Set_n$  to the convex hull $\Delta (A)$ of $A\subset \Bbb Z^n\subset \Bbb R^n$, i.e., $\rho(A)=\Delta(A)$.  Moreover, for $A,B\in  \Set_n$ the identity $\rho(A+B)=\Delta (A+B)=\Delta(A)+\Delta(B)$ holds.
\end{theorem}
A proof of Theorem \ref{theorem4.1} can be found in  \cite{Kh92}.

\subsection{The semigroup $K_n$}\label{subsec4.2}

Let $X=(\bold k ^*)^n$ be the $n$-dimensional  torus over an algebraically closed field $\bold k$ and let $\Bbb Z^n$ be the lattice of its characters. To any finite set $A\subset \Bbb Z^n$, one can associate the finite dimensional space $L_A\in K(X)$ spanned by characters from the set $A$. The multiplicative group $(\bold k ^*)^n$ naturally acts  on the field $\bold k (X)$ of rational functions on $X=(\bold k^*)^n$. The spaces $L_A$ are exactly finite dimensional spaces  invariant under this action. The set $A$ can be recovered from an invariant space as the collection of all characters of the action.

\begin{definition} The {\it semigroup}  $K_n$ is the sub-semigroup  of the semigroup $K(X)$ whose elements are the spaces $L_A$, where $A$ is a finite subset in the lattice $\Bbb Z^n$.
\end{definition}

The following lemma is obvious:

\begin{lemma} The map $\tau:K_n\to Set_n$, which sends the space $L_A\in K_n$ to the set $A\subset \Bbb Z^n$  which is considered as the element of the
semigroup $\Set_n$, provides an isomorphism between semigroups $K_n$ and $Set_n$.
\end{lemma}
One can check the following lemma (see  \cite{Kh92}).

\begin{lemma}\label{Delta} For a finite set $A\subset \Bbb Z^n$, the completion of the space $L_A$ in the semigroup $K_n$  is equal to the space $L_B$, where $B=\Delta(A)\cap \Bbb Z^n$.
\end{lemma}

Lemma \ref{Delta} relates the  algebraic operation of completion of a space $L_A\in K_n$  to the geometric operation of taking  a convex hull of a subset in a real vector space.

Lemma \ref{Delta} also explains why  the  convex hull  of the set of monomials which appear in a Laurent polynomial  $P$ with nonzero coefficients  plays such an important role in Newton polyhedra theory.


\section{ Birationally invariant an intersection index}\label{sect5}

In this section, we discuss an intersection theory
on the semigroup $K(X)$ which assigns the intersection index to an $n$-tuple of spaces $L_1,\dots,L_n$ on an $n$-dimensional irreducible variety $X$.  A birational isomorphism $\phi:X\dashrightarrow  Y$ of varieties $X$ and $Y$  provides an isomorphism $\phi^*: K(Y)\to K(X)$ between the semigroups $K(Y)$ and $K(X)$ which preserves this intersection index.

The Grothendieck group $G(X)$ is isomorphic to the group of Shokurov $b$-divisors on $X$ (see section \ref{subsec3.4},  \cite{Isk03},  \cite{K-K14}). Thus, one can define the intersection index on $K(X)$ using the intersection theory of Shokurov $b$-divisors. This approach works for algebraic varieties  over an  arbitrary algebraically closed field. In this section, we confine our consideration to algebraic varieties over the field $\Bbb C$. We introduced the  intersection  index on the group $G(X)$ in  \cite{K-K10} inspired by the BKK theorem, and we did not know about Shokurov $(b)$-divisors.
Our approach is based on simple yet powerful ideas from smooth topology. One can generalize the results of this section to algebraic varieties over arbitrary algebraically closed fields using the intersection theory of Shokurov $(b)$-divisors (see  \cite{K-K14}).

\subsection{Intersection indices on $K(X)$}
Let $X$ be an irreducible $n$-dimensional variety over $\Bbb C$. A {\it smooth open set}  $U\subset X$ is  a   Zariski open set whose complement $W=X\setminus  U $ contains  all singular points of $X$ (note that $\dim_{\Bbb C} W<n$).

\begin{definition} For any  rational vector function $\bold f=(f_1,\dots,f_n)$  on $X$ with $n=\dim_\Bbb CX$, and  for  any smooth open set  $U\subset X$,   we define the {\it number $\#\{ \bold f(x)=0|x\in U\}$   of  simple roots of $\bold f$  in $U$} as the number of points $x\in U$ such that:
\begin{enumerate}
\item  $\bold f$ is regular at $x$, and $\bold f(x)=0$;
\item the differentials $df_1,\dots, df_n$ of components of $\bold f$  are independent on the tangent space to $X$ at the point  $x$.
\end{enumerate}
\end{definition}

\begin{definition} For each space $L\in K(X)$ we define:
\begin{enumerate}
\item the {\it set  $P_L\subset X$ of poles of $L$} as the smallest Zariski closed subvariety in $X$ on whose complement $X\setminus P_L$ all functions $f\in L$ are regular (or all function from some basis  $L$ are regular);
\item  the set $B_L$  of {\it base points of $L$} as the subvariety of $X\setminus P_L$ on which vanish all functions   $f\in L$.
\end{enumerate}
A  set $U\subset X$ is {\it admissible} for a finite set  $\{L_j\}$  of spaces $L_i\in K(X)$ if it is a smooth open set, and its complement $X\setminus U$  contains the sets $P_{L_i}$ and $B_{L_i}$  of each space $L_i\in \{L_j\}$.
\end{definition}

The following theorem can be found in \cite{K-K10}.

\begin{theorem}\label{index} Let $L_1,\dots,L_n\in K(X)$ be any $n$-tuple of spaces. Then for a generic vector function $\bold f=(f_1,\dots,f_n)$ with  $f_i\in L_i$, and for any set $U$ admissible for $\{L_i\}$, the number  $\#\{\bold f=0|x\in U\}$ depends only on $L_1,\dots,L_n$; that is, it is independent of the choice of a generic vector function $\bold f$ and of the admissible set
$U$.
\end{theorem}

\begin{definition} Under the assumptions of Theorem \ref{index}, the {\it  intersection index  of  spaces $L_1,\dots,L_n\in K(X)$} is the  the number $\#\{\bold f=0|x\in U\}$ which we denote  by  $[L_1,\dots,L_n]$. We   denote by $\D(L)$ the {\it self-intersection index} of the space  $L\in K(X)$, i.e.,  $\D(L)=[L,\dots,L]$.
\end{definition}

\begin{theorem}\label{Dpolar}  The intersection index $[L_1,\dots,L_n]$ is $n$-polarization of the function $\D:K(X)\to \Bbb Q$, where $\D(L)=[L,\dots,L]$. Namely:
\begin{enumerate}

\item  $[L_1,\dots, L_n]$ is symmetric  with respect to permutations of the spaces
$L_1,\dots,L_n$;

\item $[L_1,\dots, L_n]$  is linear in each argument with respect to the multiplication.
Linearity in the first argument means that for  $L_1', L_1'', L_2,\dots,L_n\in K(X)$ we have
\begin{equation*}[L'_1 L''_1, L_2,\dots, L_n]=
[L'_1,L_2,\dots,L_n]+ [L''_1,L_2, \dots,L_n].\end{equation*}
\end{enumerate}
\end{theorem}

Theorem \ref{Dpolar} can be found in \cite{K-K10}. It implies that the self-intersection index is a homogeneous polynomial of degree $n$ on the semigroup $K(X)$.

\begin{theorem} The intersection index  $[L_1,\dots L_n]$  is:
\begin{enumerate}
\item  nonnegative, i.e., $[L_1,\dots,L_n]\geq 0$;
\item  monotone, i.e.,
if $L_1'\subset L_1, \dots, L_n'\subset L_n$, then
$ [L_1',\dots, L_n']\leq  [L_1,\dots, L_n]$.
\end{enumerate}
\end{theorem}

\begin{proof} \begin{enumerate} \item The intersection index is a nonnegative integer number.
\item Let $U\subset X$ be an admissible set for $\{L_i\}$. Then $U$ is an admissible set for $\{L_i'\}$, since $L_i'\subset L_i$.
    We have  $[L'_1,\dots, L'_n]=\#\{ \bold f'(x)=0|x\in U\}$,  where $\bold f'$
 is a generic vector function $ \bold f' =(f'_1,\dots, f'_n)$, such that $f'_i\in L'_i$. By a small perturbation of the vector function $\bold f'$, one can obtain a generic vector function $\bold f=(f_1,\dots,f_n)$, such that  $f_i\in L_i$. The number $\#\{ \bold f'(x)=0|x\in U\}$
 is less than or equal to the number $\#\{ \bold f(x)=0|x\in U\}$. Indeed,
under a small perturbation simple roots cannot disappear (they just move a little bit), but under a small perturbation new simple roots can be born.
\end{enumerate}
\end{proof}

\begin{lemma}  For any $n$-tuple of spaces  $L_1,\dots, L_n\in K(X)$ and for any $n$-tuple of one-dimensional spaces  $L'_1,\dots, L'_n$ we have:
\begin{equation*}[L_1, \dots, L_n]=[L_1 L'_1,\dots, L_nL'_n].\end{equation*}
\end{lemma}

\begin{proof} Let us show that if $\dim_{\Bbb C}L'_1=1$, then $[L'_1, L_2,\dots, L_n]=0$. Indeed, the space $L'_1$ is invertible, i.e., there is $L''_1$ such that $L'_1L''_1=\bold 1$. Thus, \begin{equation*}[L'_1, L_2,\dots, L_n] +[L''_1, L_2,\dots, L_n]=0.\end{equation*} The sum of two nonnegative numbers is equal to zero only if both of them are equal to zero.
Now the lemma follows from the multi-linearity of the intersection index.
\end{proof}

The following theorem follows immediately from definitions:

\begin{theorem}[self-intersection index  and degrees]\label{selfspace} Let $Y\subset \Bbb P(L^*)$ be the Zariski closure of the image $\Phi_L(U)$  of any admissible set  $U$  for the $n$-tuple $L,\dots,L$  under the Kodaira map $\Phi_L:X\dashrightarrow \Bbb P(L^*)$. Then:
\begin{enumerate}\item if $\dim_{\Bbb C}Y<n$, them $\D(L)=0$;
\item if $\dim_{\Bbb C}Y=n$, then $\D(L)=d(\Phi_L)d(Y)$, where $d(\Phi _L)$ is the mapping degree of the map $\Phi_L:X\dashrightarrow Y$ (i.e., is equal to
the degree of the field extension $\Phi_L^*\Bbb C(Y)\subset \Bbb C(X)$), and $d(Y)$ is the degree of the subvariety $Y$ in the projective space $\Bbb P(L^*)$.
    \end{enumerate}
\end{theorem}

\subsection{Expressing the intersection index via the function $\D$}

In this section, we   present explicit formulas for the intersection index on $G_s(X)$ and $G(X)$ via
the self-intersection index $\D(L)$ on the semigroup
$K(X)$.

\begin{lemma}[Formula for intersection index]\label{formula} The intersection index can be expressed using the function $\D$ as follows:
\begin{equation*}
[L_1,\dots, L_n]=\frac{1}{n!} \sum _{1\leq k\leq n} \Bigg( \sum _{1\leq i_1\leq \dots \leq  i_k} (-1)^{n-k} \D(L_{i_1}\cdot \ldots\cdot L_{i_k})\Bigg).\end{equation*}
\end{lemma}

\begin{proof} The lemma follows from Theorem \ref{eq1}.
\end{proof}

Since the self-intersection index $\D$ is a homogeneous polynomial on the semigroup $K(X)$, it can be extended to the Grothendieck group $G(X)$ and the extended polynomial admits an $n$-polarization.

\begin{definition} By a {\it virtual space} $\mathcal L=L/L'$, where $L, L'\in K(X)$, we mean an element  $\mathcal L$ of the Grothendieck group $G(X)$ corresponding to the pair $(L,L')$ of elements of $K(X)$, defined up to the equivalence $\sim_g$.
\end{definition}

The intersection index of virtual spaces $\mathcal L_1= L_1/L'_1,\dots, \mathcal L_n= L_n/L'_n$ we denote by $[\mathcal L_1,\dots, \mathcal L_n] $.
 The self-intersection index of the virtual space $\mathcal L=L/L'$ we denote by $\D(\mathcal L)$.  These indices  can be expressed in terms of the function $\D$ on the semigroup $K(X)$.

\begin{theorem}\label{onfracindex} The following formula holds:
\begin{equation}\label{fracindex}
   [ \mathcal L_1,\dots, \mathcal L_n] =\frac{1}{n!} \sum_{I} (-1)^{|J|} \D\bigg(\prod_{i\in I } L_i \prod _{j\in J} L'_j\bigg).
\end{equation}
\end{theorem}

Theorem \ref{onfracindex} follows from Theorem \ref{maint}.

\begin{corollary}\label{onGD} The following formula holds:
\begin{equation}\label{GD}
\D(\mathcal L)=\frac{1}{n!} \sum_{0\leq k\leq n} (-1)^{n-k}{n\choose k}\D(L^k (L')^{n-k}).
\end{equation}
\end{corollary}

Corollary \ref{onGD} follows from Corollary \ref{ongindex}.

\subsection{Intersection index of  divisors}

Let $X$ be an irreducible  projective variety and let $\Div (X)$ be the additive group of  divisors on $X$. According to the classical intersection theory of  divisors, the {\it intersection index} is defined. It assigns to an $n$-tuple of divisors $D_1,\dots,D_n\in \Div (X)$ an integer $[D_1,\dots,D_n]$ which is symmetric under permutation of divisors and multi-linear. The intersection index $[D_1,\dots,D_n]$ is the $n$-polarization of the function $\D:\Div(X)\to
\Bbb Q$ which assigns to a divisor $D$ its self-intersection index $\D(D)=[D,\dots,D]$.
The group $\Div (X)$
is the Grothendieck group of the semigroup $\Div_\mathrm{free}(X)$ of basepoint free  divisors, and the semigroup $\Div_\mathrm{free}(X)$ is isomorphic to the semigroup $K_\mathrm{reg}(X)$ (see Theorem \ref{regulargroup}).

\begin{theorem}[self-intersection on $\Div_\mathrm{free}(X)$ and degrees]\label{selfdivisor} Let $D\in \Div_\mathrm{free}(X)$, $L=L(D)$ and let  $Y\subset \Bbb P(L^*)$ be  the image $\Phi_L(X)$ of $X$ under the Kodaira map $\Phi_L:X\to \Bbb P(L^*)$.  Then:

\begin{enumerate}\item if $\dim_{\Bbb C}Y<n$, then $\D(D)=0$;
\item if $\dim_{\Bbb C}Y=n$, then $\D(D)=d(\Phi_L)d(Y)$, where $d(\Phi _L)$ is the mapping degree of the map $\Phi_L:X\to Y$ and $d(Y)$ is the degree of the subvariety $Y$ in the projective space $\Bbb P(L^*)$.
    \end{enumerate}
\end{theorem}

\begin{proof} It is easy to see  that the divisor $D$ is linearly  equivalent to the pull-back divisor $\Phi_L^* (H)$, where $H$ is a divisor on $Y\subset \Bbb P (L^*)$ obtained by a hyperplane section. The self-intersection of $D$ is equal to the number of pre-images in $X$ of points of intersection of $Y$ with a generic projective space of codimension $n$. The number of such intersection points is equal to the degree $d(Y)$ of projective subvariety $Y\subset \Bbb P(L^*)$. The number of preimages of each point of such intersection is equal to the mapping degree  $d(\Phi_L)$ of the map $\Phi_L:X\to Y$. Theorem is proven.
\end{proof}

\begin{corollary}\label{self-intersection}  The self-intersection index $ [D,\dots,D]$ of a   divisor $D\in \Div_\mathrm{free} (X)$ is equal to the self-intersection index $ [L,\dots, L]$ of the  space $L(D) \in K_\mathrm{reg} (X)$.
\end{corollary}

\begin{proof} By Theorems \ref{selfspace} and \ref{selfdivisor}, the functions $\D(L)$  and $\D (D(L))$ are given by identical formulas.
\end{proof}

\begin{theorem} The isomorphism  $\mathcal L:\Div (X)\to  G_\mathrm{reg}(X)\subset G(X)$  between the group $\Div (X)$ of  divisors on an irreducible $n$-dimensional projective variety $X$    and the group $G_\mathrm{reg}(X)$   respects the intersection indices, i.e., if  the virtual spaces $\mathcal L_1,\dots, \mathcal L_n\in G_\mathrm{reg} (X)$ correspond to the  divisors $D_1,\dots, D_n\in \Div (X)$,   then $[\mathcal L_1,\dots,\mathcal L_n]= [D_1,\dots,D_n]$.
\end{theorem}

\begin{proof} Corollary \ref{self-intersection} implies that intersection indices on the  isomorphic groups $\Div (X)$ and $G_\mathrm{reg}(X)$ are the polarizations of the homogeneous polynomials $\D(L)$ and $\D (D)$ which are identified by the isomorphism of the semigroups $K_\mathrm{reg}(X)$ and $\Div_\mathrm{free}(X)$.
\end{proof}

\begin{remark} One can deduce explicit formulas for the intersection index of   divisors  via
the self-intersection index $\D (D)$ of basepoint free  divisors.
\end{remark}

\section{Convex bodies and their  mixed volumes}\label{sect6}

Consider an $n$-dimensional real vector space $\Bbb L_n$ equipped with a volume form invariant under parallel translations. The volume of a measurable set $\Delta\subset \Bbb L_n$ is denoted by $V(\Delta)$. The collection $S$  of all compact convex bodies in $\Bbb L_n$ form an additive  semigroup with respect to Minkowski addition. The semigroup $S$ has the cancellation property.  The convex body $0$ consisting of the origin in $\Bbb L_n$ is the identity element in $S$.  This semigroup $S$ is torsion free.The Grothendieck semigroup $G_s(S)$ is naturally identified with $S$.  The natural map $\tau:S\to G(S)$  from $S$ to its Grothendieck  group $G(S)$ is an embedding.   Elements of  the  group $G(S)$ (which could be considered as the formal differences of convex bodies) are called {\it virtual convex bodies}.

The function $V:S\to \Bbb R$ which assigns to each convex body $\Delta$ its volume $V(\Delta)$ is a degree $n$ homogeneous polynomial on the semigroup $S$: Minkowski discovered that for any $m$-tuple $\Delta=(\Delta_1,\dots, \Delta_m)$  of convex bodies  $\Delta_i\in S$   the function $V_\Delta(\lambda_1,\dots, \lambda_m)= V(\lambda_1\Delta_1+\dots+ \lambda_m\Delta_m)$ is a degree $n$ homogeneous polynomial on $m$-tuples $(\lambda_1,\dots,\lambda_m)$ of nonnegative real numbers  $\lambda_i $ (in particular, it is a homogeneous  degree $n$ polynomial on $m$-tuples$(k_1,\dots,k_m)$  of nonnegative integers  $k_i$).

\begin{corollary} The function $V:S\to \Bbb R$ on the semigroup of convex bodies, which assigns to $\Delta\in S$ its volume, can be extended as a homogeneous degree $n$ polynomial on the additive  group of virtual convex bodies.
\end{corollary}

\begin{definition} The {\it  mixed volume} $MV_n$  of an $n$-tuple of convex bodies is the $n$-polarization of the volume function $V:S\to \Bbb R$  evaluated on the given n-tuple.
\end{definition}
\
\begin{lemma} The mixed volume $MV_n$  as a function  on $n$-tuples of convex bodies  $\Delta_1,\dots, \Delta_n$, $\Delta_i \subset \Bbb L_n$,  has the following properties:
\begin{enumerate}
\item (Relation with volume). On the diagonal it coincides with the volume, i.e., if $\Delta_1=\dots=\Delta_n=\Delta$, then
$MV_n(\Delta_1,\dots, \Delta_n)$ is the volume $V_n(\Delta)$ of $\Delta$.

\item  (Symmetry). $MV_n(\Delta_1,\dots, \Delta_n)$ is symmetric  with respect to permuting the bodies
$\Delta_1,\dots,\Delta_n$.

\item  (Multi-linearity). It is linear in each argument with respect to the
Minkowski sum. The linearity in the first argument means that for convex bodies $\Delta_1', \Delta_1'', \Delta_2,\dots,\Delta_n$ we have:
\begin{equation*}
MV_n(\Delta'_1 + \Delta''_1, \Delta_2,\dots, \Delta_n)= MV_n(\Delta'_1,\Delta_2,
\dots+\Delta_n)+ MV_n(\Delta''_1,\Delta_2, \dots, \Delta_n).\end{equation*}
\item (Homogeneity). It is linear homogeneous in each argument with respect to multiplication by nonnegative reals. The homogeneity in the first argument means that  for  convex bodies $\Delta_1,\Delta_2 ,\dots,\Delta_n$ and  nonnegative real number $\lambda\geq 0$ we have
$MV_n(\lambda \Delta_1,\Delta_2, \dots,\Delta_n)=\lambda MV_n(\Delta_1,\Delta_2, \dots,\Delta_n)$.

\item (Expression  in terms of volumes).
\begin{equation}\label{eq2}
MV_n(\Delta_1,\dots, \Delta_n)=\frac{1}{n!} \sum _{1\leq k\leq n} \left( \sum _{1\leq i_1\leq \dots \leq  i_k} (-1)^{n-k} V(\Delta_{i_1}+\dots+\Delta_{i_k})\right).
\end{equation}
\end{enumerate}
\end{lemma}

\begin{proof} Claims (1) -- (3) follow from the general properties of $n$-polarizations. Statements (4) due to Minkowski. Statement (5) follows from Corollary \ref{ongindex}.
\end{proof}

If $\Delta\subset \Bbb L_n$ is a convex body and $a\in \Bbb L_n$ is a point, then $\Delta+a$ is obtained from $\Delta$ by the parallel translation by the vector $a$. Since the volume form on $\Bbb L_n$  is invariant under parallel translations, the following corollary   from relation (\ref{eq2}) holds:

\begin{corollary} For any $n$-tuple of convex bodies $\Delta_1,\dots, \Delta_n$ and any $n$ vectors $a_1,\dots, a_n$ we have:
\begin{equation*}MV_n(\Delta_1, \dots, \Delta_n)=MV_n (\Delta_1 +a_1,\dots, \Delta_n+a_n)\end{equation*}
\end{corollary}

The following inequalities involving mixed volumes  are easy to prove.

\begin{theorem} The mixed volume $MV_n$  is:
\begin{enumerate}

\item  nonnegative, i.e., the mixed volume of  any $n$ -tuple $\Delta_1,\dots,\Delta_n$ of convex bodies is nonnegative;
\item  monotone, i.e.,
if $\Delta_1'\subset \Delta_1, \dots, \Delta_n'\subset \Delta_n$, then
\begin{equation*} MV_n(\Delta_1',\dots,\Delta_n')\leq  MV_n(\Delta_1,\dots,\Delta_n).\end{equation*}
\end{enumerate}
\end{theorem}

 We conclude this section with a geometric  example  of mixed volume which uses non only the volume form on $\Bbb L_n$, but also a  Euclidean structure on it.

\begin{proposition}\label{B} Let  $B$ be the unit ball centered at the origin in $\Bbb L_n$. Then for any convex body $\Delta\subset L_n$ the mixed volume $MV_n(\Delta_1,\dots, \Delta_n)$, where $\Delta_1=\dots=\Delta_{n-1}=\Delta$ and $\Delta_n=B$, is equal to the $(n-1)$-dimensional Euclidean volume of the boundary $\partial \Delta$ of $\Delta$ multiplied by $\frac{1}{n}$.
\end{proposition}

\subsection{Geometric inequalities involving mixed volumes}

Mixed volumes appear in many geometric inequalities. In this section, we state some of them. Probably, the famous Alexandrov--Fenchel inequality (see below) is the most important.  The classical isoperimetric inequality which was known to ancient Greeks is a very special  case of that inequality. We  state the  Minkowski criterion  of vanishing of   mixed volumes and  the Brunn--Minkowski inequality on volumes of convex bodies.

\begin{definition} A set $\{\Delta_1,\dots,\Delta_k\}$ of $k$
convex bodies in  $\Bbb L_n$ elements is {affinely dependent} if   for some nonempty subset $J\in \{1,\dots,k\}$  the body $\Delta_J=\sum_{j\in J} \Delta_j$ has dimension smaller  than the number $|J|$ of elements in the set $J$.
\end{definition}

\begin{theorem}[Minkowski criterion]\label{Mink}
 Let $\Delta_1,\dots,\Delta_k$ be a set of $k\leq n$  convex
bodies in $\Bbb L_n$.
Then the mixed volume $MV_n(\Delta_1,\dots, \Delta_k, \Delta_{k+1},\dots, \Delta_n)$ vanishes
for any collection $\Delta_{k+1},\dots,\Delta_n$ of convex bodies in $\Bbb L_n$  if and only if the set
$\Delta_1,\dots, \Delta_k$ is affinely dependent.
\end{theorem}

\begin{theorem}[Brunn--Minkowski inequality]\label{B-M} If $\Delta_1,\Delta_2$ are convex bodies in $\Bbb L_n$, then
\begin{equation}\label{eq3}
V ^{\frac{1}{n}}(\Delta_1) + V ^{\frac{1}{n}}(\Delta_2) \leq V ^{\frac{1}{n}}(\Delta_1+\Delta_2).
\end{equation}

\end{theorem}

The Brunn--Minkowski inequality has visual  geometric proofs. It was discovered  by Brunn in an equivalent, but very different form. Minkowski found its present form  and described  conditions under which the inequality becomes an equality.

\begin{theorem}[Alexandrov--Fenchel inequality]\label{A-F}
For any $n$-tuple of convex bodies
$\Delta_1,\Delta_2,\dots,\Delta_n$ in $\Bbb L_n$ the following inequality holds:
\begin{equation}\label{eq4}
MV_n^2(\Delta_1,\Delta_2, \Delta_3,\dots,\Delta_n)\leq
MV_n(\Delta_1,\Delta_1, \Delta_3,\dots,\Delta_n)\times
MV_n(\Delta_1,\Delta_1, \Delta_3,\dots,\Delta_n).
\end{equation}

\end{theorem}

The Alexandrov--Fenchel inequality is one of the most  general inequalities between mixed volumes of convex bodies. It has many corollaries. Below are two examples.

\begin{corollary}\label{AFcor1} For any $2\leq m\leq n$ and for any   $n$-tuple of convex bodies $\Delta_1,\dots, \Delta_n$ we have:
\begin{equation}\label{cor1}
 \prod_{1\leq i\leq m} MV_n (\Delta_i,\dots,\Delta_i, \Delta_{m+1}, \dots, \Delta_n)\leq MV_n (\Delta_1,\dots,\Delta_n)^m.
 \end{equation}
\end{corollary}

For $m=2$, inequality (\ref{cor1}) coincides with the inequality (\ref{eq4}). For $m=n$, inequality (\ref{cor1}) is symmetric in  convex bodies $\Delta_1,\dots,\Delta_n$.

\begin{corollary}\label{AFcor2}  For any $2\leq m\leq n$ and for any collection of convex bodies $\Delta_1, \Delta_2$, $\Delta_{m+1},\dots,\Delta_n$ we have:
\begin{equation}\label{cor2}\begin{aligned}
MV_n^{\frac{1}{m}}(\Delta_1,\dots,\Delta_1,\Delta_{m+1},\dots,\Delta_n)+ MV_n^{\frac{1}{m}}(\Delta_2,\dots,\Delta_2,\Delta_{m+1},\dots,\Delta_n)\\
\leq MV_n^{\frac{1}{m}}(\Delta_1+\Delta_2,\dots,\Delta_1+\Delta_2,\Delta_{m+1},\dots,\Delta_n).\end{aligned}
\end{equation}
\end{corollary}

For $m=n$, inequality (\ref{cor2}) coincides with inequality (\ref{eq3}). Thus, the  Brunn--Minkowski inequality is  a corollary of the Alexandrov--Fenchel inequality. On the other hand, all known geometric proofs of the Alexandrov--Fenchel inequality are based on the Brunn--Minkowski inequality (all of  these proofs are rather tricky). In fact,  Theorem \ref{B-M} and  Theorem \ref{A-F}  follow from  algebraic versions of these theorems (see section \ref{sect8}).  Besides pure algebraic arguments, our proof of these algebraic versions  makes use of Theorem\ref{B-M} for the case $n=2$. The case $n=2$ is both very special and simple algebraically and geometrically.  We focus on it in the next section.

\subsection{Inequalities in dimension two}

The area of a convex body in the standard plane $\Bbb R^2$ is a homogeneous  polynomial of degree two on the semigroup of planer convex bodies. We show below that
a  homogeneous polynomial $Q:S\to \Bbb R$  of degree two on any  semigroup $S$  satisfies the Brunn--Minkowski type inequality for  $x,y\in S$ if and only if it satisfies the  Alexandrov--Fenchel type inequality for $x,y\in S$ (see Theorem \ref{B-MA-F}).
In addition, the study of  homogeneous  polynomials  of degree two on a semigroup $S$ can be reduced to the study of quadratic forms on the real vector space $G(S)\otimes_{\Bbb Z} \Bbb R$.

\begin{definition}\label{polynomial B-Mxy}
Let $Q:S\to \Bbb R$ be a homogeneous polynomial of degree two, let $B:S^2\to \Bbb R$ be its $2$-polarization. Then for a pair of points  $x,y\in S $ the polynomial $Q$ satisfies:
\begin{enumerate}
\item the Brunn--Minkowski type inequality  if $Q(x)\geq 0$, $Q(y)\geq 0$, $Q(x+y)\geq 0 $ and
\begin{equation}\label{eq5}
Q^{\frac{1}{2}}(x) + Q^{\frac{1}{2}}(y) \leq Q^{\frac{1}{2}}(x+y);
\end{equation}
\item the  Alexandrov--Fenchel  type inequality if $Q(x)\geq 0$, $Q(y)\geq 0$,  $B(x,y)\geq 0$ and
\begin{equation}\label{eq6}
  Q(x)Q(y)\leq B^2(x,y).
\end{equation}
\end{enumerate}
\end{definition}

The following theorem is  easy to prove:

\begin{theorem}\label{B-MA-F} A homogeneous polynomial of degree two   satisfies the Brunn--Minkowski type inequality for $x, y \in S$ if and only if it
satisfies the Alexandrov--Fenchel  type inequality for $x, y\in S$.
\end{theorem}

\begin{proof} Let us show that  if  $Q$ satisfies the Brunn--Minkowski  type inequality for $x,y\in S$, then it satisfies the Alexandrov--Fenchel inequality for $x$ and  $y$.
Indeed,  by squaring  both sides of  (\ref{eq5} ), we  obtain $Q(x) + 2 Q(x)^{\frac{1}{2}}Q(y)^{\frac {1}{2}} + Q(y)\leq Q(x)+2B(x,y)+ Q(y)$, or $Q(x)^{\frac{1}{2}}Q(y)^{\frac {1}{2}}\leq B(x,y)$. Thus, $B(x,y)$ is nonnegative. By squaring  both sides of the previous inequality, we obtain  $Q(x)Q(y)\leq B^2(x,y)$.  Theorem is proven in one direction. Its proof in the opposite direction is similar.
\end{proof}

A homogeneous polynomial $Q$ on $S$ can be induced from a homogeneous polynomial  on the real vector space $S_{\Bbb R}= G(S)\otimes_{\Bbb Z} \Bbb R$.  The  inequality $B^2(x,y)\geq Q(x)Q(y)$  for two independent vectors $x,y\in S_{\Bbb R}$ means that the restriction of the quadratic form $Q$ to the plane spanned by $x$ and $y$ can not be positive or negative  definite. If $Q(x)>0$ at some point $x\in S$, then $Q(y)\leq 0$ on the hyperplane in $S_{\Bbb R}$ orthogonal to $x$, i.e., $B(x,y)=0$.

Our proofs of the  Alexandrov--Fenchel inequalities for mixed volumes and of their version  in intersection theory are based on the Brunn--Minkowski inequality for 2-dimensional convex bodies.
This inequality can be proved using two different approaches.

First, the Brunn--Minkowski inequality in any dimension is not difficult to establish.

Second, for $n=2$, the Alexandrov--Fenchel inequality can be proved directly by slightly modifying one of the classical proofs of the isoperimetric inequality (see the proof of Theorem A in \cite{Kh-T14}). By Theorem \ref{B-MA-F}, the Brunn--Minkowski inequality in dimension $n=2$ follows from the Alexandrov--Fenchel inequality.

Let us return to classical geometry and prove the isoperimetric inequality.

\begin{corollary} The area $V(\Delta)$ of a planar  convex body $\Delta$ satisfies the following inequality: \begin{equation*} V(\Delta)\leq \frac{1}{4\pi} l(\partial \Delta)^2,\end{equation*}
where $l(\partial \Delta)$ is the length of the boundary of $\Delta$. Moreover, if $\Delta$ is a ball, then the inequality becomes equality.
\end{corollary}

\begin{proof} Corollary follows from  Proposition \ref{B} and from  inequality (\ref{eq4}) for $n=2$; $\Delta_1=\Delta$; $\Delta_2=B_1$, where $B_1$ is the unit ball.
  \end{proof}

\section{Newton--Okounkov bodies and intersection theory}\label{sect7}
The theory of Newton--Okounkov bodies (or, shortly,  of NO bodies) is based on the theory of semigroups of lattice points, i.e., of sub-semigroups of the group $\Bbb Z^n$ (see  \cite{K-K12} and  \cite{K-K08},  \cite{LM09}).

For any faithful $\Bbb Z^n$-valuation on the field of rational functions on $n$-dimensional irreducible variety $X$, the theory of NO bodies assigns to a each space $L\in K(X)$  a convex body $\Delta(L)\subset \Bbb R^n$. In particular, it relates  the self-intersection index of a space  $L\in K(X)$ and the volume of its NO body $\Delta(L)\subset \Bbb R^n$.

\subsection{NO bodies of elements of the semigroup $K(X)$}

Consider the lattice  $\Bbb Z^n$, equipped with a total order $\succ$, which  is compatible with the addition (i.e., if $a\succ b$, then $a+c>b+c$). For example, one can consider the  lexicographic  order.

 A {\it  $\Bbb Z^n$-valuation} $v$ on the field  $\bold k(X)$  of rational functions on an irreducible $n$-dimensional variety $X$ over an algebraically closed field $\bold k$  is {\it faithful} if the valuation map
$v:\bold k(X)\setminus \{0\} \to \Bbb Z^n$ is  surjective. Let us give an example of a faithful $\Bbb Z^n$-valuation $v$
on $n$-dimensional irreducible  variety $X$ over $\Bbb C$.

\begin{example}\label{exval}  Let $a$ be a smooth point in $X$ and let $x_1,\dots, x_n$ be a  system of coordinates in a neighborhood of $a$, such that $x_1(a)=\dots=x_n(a)=0$. In the example, the valuation   $v(f)$  of a  regular function  at $a$   is equal to $\bold m =(m_1,\dots, m_n)\in \Bbb Z^n$, where $\bold x^\bold m= x_1^{m_1}\cdot \ldots \cdot x_n^{m_n}$
is the smallest  monomial  in the  lexicographic order which appears with nonzero coefficient in the Taylor series of the function  $f$ at the point $a$. The valuation $v$ of a ratio $f/g$  of  functions $f,g\in \Bbb C(X)$, which are regular at the point  $a$ (and $f,g$ are  not identically equal to  zero) is equal, by  definition to $v(f)-v(g)$.
\end{example}

 Elements $\{L(i)\}$,  $i=1,\dots,n,\dots$  of the semigroup  $K(X)$ form  a {\it multiplicative sequence} if for any natural numbers $n, k$ the inclusion $L(n) L(k)\subset L(n+k)$ holds. Using a  $\Bbb Z^n$-valuation $v$, to any multiplicative sequence $\{L(i)\}$ one can  assigns   a convex body $\Delta (\{L(i)\}) \subset  \Bbb R^n \supset Z^n$ (for details see  \cite{K-K12}) which is called  the Newton--Okounkov body (NO body)  of the multiplicative sequence $\{L(i)\}$ with respect to the chosen valuation $v$.
The lattice $\Bbb Z^n\subset \Bbb R^n$  endows the space $\Bbb R^n$ with the volume form invariant under the parallel translations. The measure of a convex body $\Delta \subset \Bbb R^n$ is denoted by $V(\Delta $).
 We assume  below   that a faithful valuation $v$ on $\bold k (X)$ is chosen and fixed.

\begin{definition}  The {\it Newton--Okounkov body} $\Delta( L)$  of an element $L\in K(X)$ is defined as the Newton--Okounkov body of the multiplicative semigroup $L^i$ generated by $L$.
\end{definition}

 For the purposes of the paper, the most important result of  \cite{K-K12} is the following theorem:

\begin{theorem}[NO boby of $L\in K(X)$]\label{main}
The following relations hold:

\begin{enumerate}
\item for any $L \in K(X)$ we have:  $[L,\dots, L] = n! V(\Delta(L))$;

\item for any $L_1,L_2 \in K(X)$ we have:  $\Delta(L_1L_2) \supset  \Delta(L_1) + \Delta(L_2)$.
\end{enumerate}
\end{theorem}

Statement (1)  relates the self-intersection index of $L\in K(X)$ with its NO body $\Delta(L)$.  Statement (2)   implies all geometric type inequalities for the semigroup $K(X)$.
 Let us sketch its  proof. For $\mathcal L_1, \mathcal L_2\in K(X)$  and $\mathcal L_3=\mathcal L_1 \cdot \mathcal L_2$, let $A_1, A_2,A_3\subset \Bbb Z^n$ be the images of the sets  $\mathcal L_1\setminus\{0\}$, $\mathcal L_2\setminus\{0\}$, $\mathcal L_3\setminus\{0\}$ under the valuation map. Then the inclusion  $A_3\supset A_1 +A_2$ holds. Indeed, for any $a_1\in A_1$ and $a_2\in A_2$  there exist nonzero elements  $f_1\in \mathcal L_1$ and $f_2\in \mathcal L _2$ such that $v(f_1)=a_1$ and $v(f_2)=a_2$. For $f_3=f_1f_2 \in \mathcal L_3$ we have  $v(f_3)=a_1+a_2$, which proves the inclusion $A_3\supset A_1+A_2$.

The statement (2) follows from these inclusions for the terms of  sequences of spaces that define the NO bodies $\Delta(L_1)$ and $\Delta(L_2)$.

\subsection{The BKK theorem and  formulas for the intersection index}\label{subsec7.2}
In this section, we present a proof of the BKK (Bernstein--Koushnirenko--Khovanskii) theorem based on NO bodies. We also present   formulas expressing the intersection index via NO bodies. These results  are based on the first statement of the Theorem \ref{main}.

Koushnirenko's theorem is a very particular case of the first statement of Theorem \ref{main} related to  the semigroup $K_n$ of $(\Bbb C^*)^n$-invariant spaces on the
  complex torus $X=(\Bbb C^*)^n$ (see section \ref{subsec4.2}). Let $\Bbb Z^n$ be the lattice of  characters of $(\Bbb C^*)^n$.  With any finite set $A\subset \Bbb Z^n$ one can associate the finite dimensional space $L_A\in K(X)$ spanned by the characters from the set $A$.
By definition, the {\it Newton polytope} $\Delta(A)$ of $L_A$ is the convex hull of the set $A\subset \Bbb Z^n \subset \Bbb R^n$.

\begin{theorem}[Koushnirenko's theorem]\label{Ku} For $n$ generic functions $f_1,\dots,f_n\in L_A$ the number of solutions in $(\Bbb C^*)^n$ of the system $f_1=\dots=f_n=0$ is equal to $n! V(\Delta(A))$.
\end{theorem}

To reduce  Theorem \ref{Ku} to Theorem \ref{main}, it is enough to  present an example of a faithful $\Bbb Z^n$-valuation $v$ on the field $\Bbb C(X)$ for $X=(\Bbb C^*)^n$, such that the NO body of each space $L_A$ with respect to $v$ is equal to $\Delta(A)$. It is easy to see that such equalities hold if and only if for each character
$\bold x^{\bold m}$  the equality $v(\bold x^{\bold m})= \bold m $ holds.

\begin{example}\label{exBKK} The torus $(\Bbb C^*)^n$ is bi-rationally  equivalent to the affine space $\Bbb C^n$. Consider the valuation $v$
from Example \ref{exval} in which the point $a\in \Bbb C^n$ is the origin in $\Bbb C^n$ and $x_1,\dots,x_n$ is the standard system of coordinates. Obviously, for any monomial $\bold x^\bold m$ the identity $v(\bold x^\bold m)=\bold m\in \Bbb Z^n$ holds.
\end{example}

Thus, Koushnirenko's theorem is proven.

Let us show that the intersection index of elements from the semigroup $K(X)$ can always  be expressed via NO bodies of elements of $K(X)$.

\begin{theorem}\label{ONBKK} The intersection index of $L_1,\dots,L_n\in K(X)$  can be expressed using  volume of NO bodies:
\begin{equation}\label{ONBKKe} [L_1,\dots, L_n]= \sum _{1\leq k\leq n} \Bigg( \sum _{1\leq i_1\leq \dots \leq  i_k} (-1)^{n-k} V(\Delta (L_{i_1}\cdot \ldots\cdot L_{i_k}))\Bigg).\end{equation}
\end{theorem}

\begin{proof} Theorem \ref{ONBKK} follows from Theorem \ref{main}  and Lemma \ref{formula}.
\end{proof}

Note that  the right-hand side of (\ref{ONBKKe}) is not equal to the  mixed volume of the NO bodies $\Delta(L_1), \dots, \Delta (L_n)$ multiplied by $n!$,  because the inclusion in statement (2) of Theorem \ref{main} is, in general, strict.

The  BKK theorem is a very special case of Theorem \ref{ONBKK}. Below we use notations introduced before Theorem \ref{Ku}.
Let $A_1,\dots,A_n$ be an $n$-tuple of finite  subsets in  $\Bbb Z^n$ and let $L_{A_1},\dots, L_{A_n}$ be the corresponding $n$-tuple of elements $K(X)$  for $X= (\Bbb C^*)^n$.

\begin{theorem}[BKK theorem]\label{BKK} For $n$ generic functions $f_1\in L_{A_1},\dots,f_n\in L_{A_n}$ the number of solutions in $(\Bbb C^*)^n$ of the system $f_1=\dots=f_n=0$ is equal to \begin{equation*}n!M V_n(\Delta(A_1),\dots, \Delta(A_n)).\end{equation*}
\end{theorem}

\begin{proof} Let us apply Theorem \ref{ONBKK} to the valuation $v$ from Example \ref{exBKK}.  We have:
\begin{equation*}[L_{A_1},\dots, L_{A_n}]= \sum _{1\leq k\leq n} \Bigg( \sum _{1\leq i_1\leq \dots \leq  i_k} (-1)^{n-k} V(\Delta (L_{A_{i_1}}\cdot \ldots\cdot L_{A_{i_k}}))\Bigg).\end{equation*}
Note that for any   finite sets $A,B\in \Bbb Z^n$ the identities  $\Delta(L_A)=\Delta(A)$ and $\Delta (L_A L_B)=\Delta(A+B)$ hold. Thus, we have
\begin{equation*}
[L_{A_1},\dots, L_{A_n}]=\sum _{1\leq k\leq n} \Bigg( \sum _{1\leq i_1\leq \dots \leq  i_k} (-1)^{n-k} V(\Delta (A_{i_1})+\dots+\Delta(A_{i_k}))\Bigg).\end{equation*}
By formula (\ref{eq1}), for the polarization applied of the function $n!V$, we obtain the desired identity.
\end{proof}

A simpler version of the same proof of the BKK theorem, which does not use Theorem \ref{main}, can be found in   \cite{Kh92}.

Let us present   formulas expressing the intersection index of virtual spaces and self-intersection index for virtual spaces via NO bodies

\begin{theorem}\label{virtualNO} For $\mathcal L_1=L_1/L'_1,\dots, \mathcal L_n= L_n/L'_n$, the following formula holds:
\begin{equation}\label{fvirtualNO}
[\mathcal L_1,\dots, \mathcal L_n] = \sum_{I} (-1)^{|J|}V\Bigg(\Delta\Bigg(\prod_{i\in I } L_i \prod _{j\in J} L'_j\Bigg)\Bigg).
\end{equation}
\end{theorem}

Theorem \ref{virtualNO} follows from Theorem \ref{onfracindex}.

\begin{corollary}\label{selfNo} For $\mathcal L=L_1/L_2$, the following formula holds:
\begin{equation}\label{fselfNO}
\D(\mathcal L) = \sum_{0\leq k\leq n} (-1)^{n-k}{n\choose k} V(\Delta (L_1^k L_2^{n-k})).
\end{equation}
\end{corollary}

Corollary \ref{selfNo}  follows from Corollary \ref{ongindex}.

\subsection{NO bodies and geometric type inequalities for the semigroup $K(X)$}

Newton--Okounkov bodies allow us to prove inequalities for the birationally invariant intersection index on the semigroup $K(X)$ on an irreducible $n$-dimensional algebraic variety $X$ (possibly non complete and singular). Proofs of the inequalities presented below  are based on the second statement of Theorem \ref{main} and on the Brunn--Minkowski inequality from convex geometry (in fact, in proofs one can use the Brunn--Minkowski inequality only for two dimensional convex bodies).

\begin{theorem}[Brunn--Minkowski type inequality]\label{aBM} Let $X$ be an irreducible $n$-dimensional  variety,  $L_1, L_2 $ be arbitrary elements in $ K(X)$ and let $L_3$ be their product,$L_3 = L_1L_2$.  Then the following inequality holds:
 $[L_1,\dots, L_1]^{1/n} + [L_2,\dots, L_2]^{1/n} \leq  [L_3,\dots, L_3]^{1/n}$.
 \end{theorem}

\begin{proof} Let $\Delta(L_1)$, $\Delta(L_2)$, and $\Delta(L_3)$ be the NO bodies of the spaces $L_1$, $L_2$ and $L_3=L_1L_2$.  By Theorem \ref{main}, we have the following inclusion: $\Delta (L_1)+\Delta(L_2) \subset \Delta (L_3)$. By  the Brunn--Minkovsky inequality, we have $V^{1/n} ( \Delta(L_1)) +  V^{1/n} (\Delta(L_2))\leq V^{1/n} (\Delta(L_3))$. By Theorem \ref{main} for $i=1,2,3$, we have $[L_i,\dots L_i]=n!V( \Delta (L_i))$. These relations imply the theorem.
\end{proof}

\begin{theorem}[Hodge type inequality]\label{hodge} Let $X$ be an irreducible  algebraic surface, let  $L_1, L_2 $ be arbitrary elements in $ K(X)$ and let $L_3$ be their product,$L_3 = L_1L_2$.  Then the following inequality holds:
$[L_1, L_1][L_2, L_2] \leq  [L_1, L_2]^2$.
\end{theorem}

\begin{proof} By Theorem \ref{aBM},   $[L_1,L_1]^{1/2}+[L_2,L_2]^{1/2}\leq [L_1L_2]^{1/2}$. By Theorem \ref{B-MA-F}, this inequality is equivalent to the inequality  $[L_1, L_1][L_2, L_2] \leq  [L_1,L_2]^2$.
\end{proof}

Note that in Theorems \ref{aBM} and  \ref{hodge}   the  irreducibility assumption can not be dropped:  for reducible surfaces the Brunn--Minkowski type inequality and the  Hodge type inequality in general do not hold.

Recall that an element $L\subset  K(X)$ is  called {\it very big} if  the Kodaira
rational map $\Phi_L $ provides  a birational isomorphism between $X$ and its image $Y=\Phi_L(X)$.

\begin{lemma}[Algebraic version of Alexandrov--Fenchel  inequality]\label{weakKht} If $L_3,\dots,L_n$ are very big  subspaces, then
\begin{equation*}[L_1, L_2, L_3,\dots, L_n]^2 \geq [L_1,L_1, L_3, \dots, L_n][L_2,L_2, L_3, \dots, L_n].\end{equation*}
\end{lemma}

Let us sketch a  proof of Lemma \ref{weakKht} (more complete presentation  can be found in  \cite{K-K12}).
Let $U\subset X$ be an admissible set for $L_1,\dots,L_n$ and let $\bold f=(f_3,\dots,f_n)$ be a generic vector function such that $f_3\in L_3,\dots, f_n\in L_n$. Then the variety $X_{\bold f}$ defined in $U$ by $f_3=\dots=f_n=0$ is a smooth surface.

The assumption that the spaces $L_3,\dots,L_n$ are very big allows us to apply  the  Bertini--Lefschetz theorem on irreducibility and to show that the surface $X_{\bold f}$ is irreducible. Lemma \ref{weakKht} follows from the Hodge type inequality for the restriction of spaces $L_1,L_2$ to  the  irreducible surface $X_{\bold f}$.

\begin{remark} In Lemma \ref{weakKht},   the assumption that the spaces $L_1,\dots,L_n$ are very big can be dropped (see Corollary \ref{algAF}). Note that Lemma \ref{weakKht} and Theorem \ref{strongKht} can be proven using only the Brunn--Minkowski inequality  in dimension two. All other geometric type inequalities follow from inequality (\ref{nefaf}).
\end{remark}

\subsection{Inequalities for psef  type elements in the group   $G(X)$}

In the Grothendieck group $G(X)$ of the semigroup  $K(X)$, one can define nef elements. For such elements the Alexandrov--Fenchel type inequality holds. So all inequalities from convex  geometry that follow from the Alexandrov--Fenchel inequality (see Corollaries \ref{AFcor1} and \ref{AFcor2}) also hold for nef elements in $G(X)$.

\begin{definition} An element $\mathcal L$ of the Grothendieck group  $G(X)$ is called
\begin{enumerate}
\item   {\it a very big element} if it can be represented by a  very big element element $L\in K(X)$;
\item {\it a big element} if there is a natural number $k$ such that $\mathcal L^k$ is a very big element;

\item   {\it a  psef type element} if there is a   big element  $\mathcal L'$ such that  for every natural number $q$ the element $\mathcal L^q\mathcal L'$ is  big.
    \end{enumerate}
    \end{definition}

\begin{lemma} Very big elements, big elements and psef type elements  in the Grothendieck  group $G(X)$ form  sub-semigroups.
\end{lemma}

\begin{proof} If $\mathcal L_1$ and $ \mathcal L_2$ are represented by very big spaces $L_1$ and $L_2$, then $\mathcal L_1 \mathcal L_2$ is represented by a very big space $L_1 L_2$.
If $\mathcal L_1^k$ and $\mathcal L_2^m$ are very big elements, then $(\mathcal L_1 \mathcal L_2)^{km}$ is a very big element. If $\mathcal L^q_1  \mathcal L_1'$ and $\mathcal L_2^q \mathcal L_2'$ are big elements, then $(\mathcal L_1 \mathcal L_2)^q (\mathcal L_1' \mathcal L_2')$ are big elements.
\end{proof}

 \begin{lemma} Any element $L$ in  the semigroup $K(X)$ represents a psef type element in its Grothendieck group $G(X)$.
 \end{lemma}

\begin{proof} If  $L,L'\in K(X)$  and $L'$ is a very big space, then  for any natural number $q$ the product  $L^qL'$ is also  a very big space.
\end{proof}

Let us prove the Alexandrov--Fenchel type  inequality  for psef type elements in $G(X)$.

\begin{theorem}\label{strongKht} If $\mathcal L_1, \mathcal L_2, \mathcal L_3,\dots,\mathcal L_n$ are
psef type elements in the group $G(X)$, then
\begin{equation}\label{nefaf}
[\mathcal L_1, \mathcal L_2,\mathcal  L_3,\dots, \mathcal L_n]^2 \geq [\mathcal L_1, \mathcal L_1, \mathcal  L_3, \dots,\mathcal L_n] [ \mathcal L_2, \mathcal L_2,   \mathcal L_3, \dots, \mathcal  L_n].
\end{equation}
\end{theorem}

\begin{proof} If each  $\mathcal L_i$ is a very big element in $G(X)$, then theorem follows from Lemma \ref{weakKht}, since the intersection index respects the equivalence $\sim$ in $K(X)$.

Assume that $\mathcal L_1,\dots,\mathcal L_n$ are big elements, i.e., for some natural number $k_1,\dots,k_n$ the elements $\mathcal L_1^{k_1},\dots,\mathcal L_n^{k_n}$ are very big  elements. Put $k=k_1\cdot \ldots \cdot k_n$. Then elements $\mathcal L_1^k,\dots,\mathcal L_n^k$ are very big spaces. Thus, for the $n$-tuple of spaces $\mathcal L_1^k,\dots,\mathcal L_n^k$  inequality (\ref{nefaf}) holds. It implies inequality (\ref{nefaf}) for the $n$-tuple $\mathcal L_1,\dots,\mathcal L_n$, since the intersection index is multi-linear.

Assume that  $\mathcal L_i$ are  psef type elements, i.e, there are  big elements $\mathcal L_i'$ such that for any natural number $q$ all elements $\mathcal L_i^q\mathcal L_i'$ are  big. Put $\mathcal L=\mathcal L_1'\cdot \ldots \cdot \mathcal L_n'$. Then all elements $\mathcal L_i^q\mathcal L$ are  big.
Instead of elements
$\mathcal L_1,...,\mathcal L_n$,
let us consider the elements  $\mathcal L_1(q) =\mathcal L_1^q \mathcal L$ $,\dots$,   $\mathcal L_n(q) =\mathcal L_n^q \mathcal L$
for every  natural number  $q$. For the element $\mathcal L_1(q),\dots, \mathcal L_n(q)$ the Alexandrov--Fenchel type inequality holds, i.e., for every natural number $q$, we  have $A(q)\geq B(q)$, where $A(q)$ and $B(q)$ are the following polynomials in $q$:
\begin{equation*}A(q)=[\mathcal L_1(q),\mathcal L_2(q),\mathcal  L_3(q),\dots,\mathcal  L_n(q)]^2; \end{equation*}  \begin{equation*}B(q)= [\mathcal L_1(q), \mathcal L_1(q), \mathcal  L_3(q), \dots,\mathcal L_n(q)][\mathcal L_2(q), \mathcal L_2(q), \mathcal L_3 (q), \dots, \mathcal  L_n (q)].\end{equation*}

The polynomials $A(q)$ and $B(q)$ have the following  leading terms  $A_{2n}$, $B_{2n}$:  \begin{equation*}A_{2n}=q^{2n}  [\mathcal L_1, \mathcal L_2, \mathcal L_3,\dots, \mathcal L_n]^2;\end{equation*}
\begin{equation*}B_{2n}=q^{2n} [\mathcal L_1,\mathcal L_1, \mathcal  L_3, \dots,\mathcal L_n] [\mathcal L_2,\mathcal L_2, \mathcal  L_3, \dots, \mathcal  L_n].\end{equation*}  The inequalities $A(q)\geq B(q)$  imply that $A_{2n}\geq B_{2n}$ which proves inequality (\ref{nefaf}).
\end{proof}

Corollaries \ref{AFcor1} and \ref{AFcor2} provide examples of statements which formally follow from the Alexandrov--Fenchel inequality. Similarly, the following analogues hold for intersection indices.

\begin{corollary} For any $2\leq m\leq n$ and for any   $n$-tuple of  psef type elements  $\mathcal L_1,\dots, \mathcal L_n \in G(X)$, we have:

\begin{equation}\label{alcor1}
 \prod_{1\leq i\leq m} [\mathcal L_i,\dots,\mathcal L_i, \mathcal L_{m+1}, \dots, \mathcal L_n]\leq [\mathcal L_1,\dots,\mathcal L_n]^m.
 \end{equation}
\end{corollary}

For $m=2$, inequality (\ref{alcor1}) coincides with inequality (\ref{nefaf}). For $m=n$, inequality (\ref{alcor1}) is symmetric in the psef type elements  $\mathcal L_1,\dots,\mathcal L_n$.

\begin{corollary}  For any $2\leq m\leq n$ and for any collection of psef type elements  $\mathcal L_1, \mathcal L_2$, $\mathcal L_{m+1},\dots,\mathcal L_n$, we have:
\begin{equation}\label{nefcor2}\begin{aligned}
[\mathcal L_1,\dots,\mathcal L_1,\mathcal L_{m+1},\dots,\mathcal L_n]^{\frac{1}{m}}+ [\mathcal L_2,\dots,\mathcal L_2,\mathcal L_{m+1},\dots,\mathcal L_n]^{\frac{1}{m}}\\
\leq  [\mathcal L_1\mathcal L_2,\dots,\mathcal L_1\mathcal L_2,\mathcal L_{m+1},\dots,\mathcal L_n]^{\frac{1}{m}}.\end{aligned}
\end{equation}
\end{corollary}

For $m=n$, inequality (\ref{nefcor2})
provides the  Brunn--Minkowski  type inequality for psef type  elements in $G(X)$.

\begin{corollary}\label{algAF} The inequalities (\ref{nefaf}), (\ref{alcor1}) and (\ref{nefcor2}) hold for every elements $L_1,\dots, L_n\in K(X)$,
since they represent psef type elements in the group $G(X)$. Thus, Lemma \ref{weakKht} holds for any $n$-tuple of spaces $L_1,\dots, L_n\in K(X)$ (without the  assumption that these spaces  are very big spaces).
\end{corollary}

\subsection{Inequalities for nef  divisors}
Let $X$ be an irreducible  projective  $n$-dimensional variety.

\begin{definition} A divisor  $D\in \Div(X)$   is
  {\it  nef divisor} if there is an ample divisor $D$  such that  for every natural number $q$ the divisor  $qD + D'$ is  ample.
    \end{definition}

One can show that very ample divisors, ample divisors  and  nef divisors form  additive sub-semigroups in the group $\Div(X)$ of  divisors.

 \begin{lemma} Any basepoint free  divisor  $D$ is a nef divisor.
 \end{lemma}

\begin{proof} If  $D,D'$  are basepoint free  divisors, and  $D'$ is a very ample divisor, then  for any natural number $q$ the divisor  $qD+D'$  is also  a very ample divisor.
\end{proof}

\begin{lemma} The natural homomorphism  $\mathcal L:\Div(X)\to G(X)$ maps  the semigroups of  very ample divisors,  ample divisors and nef divisors  respectively to the semigroups of  very big  elements,  big elements and  psef type elements.
\end{lemma}

\begin{proof} If the Kodaira  map $\Phi_{L(D)}$ is an embedding, then $\Phi_{L(D)}$ provides a birational isomorphism between $X$ and $\Phi_{L(D)} (X)$.
If a divisor $k D$ is very ample, then the element $\mathcal L(D)^k\in G(X)$ can be represented by a very big space, thus, $\mathcal L(D)$ is a big element in $G(X)$. If a divisor $D$ is nef, then for some ample divisor $D'$  for any $q>0$ the divisor $qD+D'$ is ample, thus, the element $\mathcal L(qD+D')=(\mathcal L(D))^q \mathcal L(D')$ is a big element in $G(X)$. Since $\mathcal L( D')$ is a big element in $G(X)$, it implies that $\mathcal L (D)$ is a psef type element in $G(X)$.
\end{proof}

Since the homomorphism $\mathcal L:\Div(X)\to G(X)$ preserves the intersection indices, we obtain the following theorem:

\begin{theorem}\label{KhT-1} If $ D_1, D_2,  D_3,\dots,  D_n$ are
nef  divisors on an irreducible projective variety $X$,  then the following inequality holds:
\begin{equation}\label{divAF}
[D_1, D_2, D_3,\dots, D_n]^2 \geq [D_1, D_1,  D_3, \dots, D_n] [D_2, D_2,   D_3, \dots, D_n].
\end{equation}
\end{theorem}

\begin{corollary}\label{KhT-2} For any $2\leq m\leq n$ and for any $n$-tuple of  nef  divisors  $ D_1,\dots,  D_n \in \Div (X)$, we have:
\begin{equation}\label{divalcor1}
 \prod_{1\leq i\leq m} [ D_i,\dots, D_i,  D_{m+1}, \dots,  D_n]\leq [ D_1,\dots, D_n]^m.
 \end{equation}
\end{corollary}

Theorem \ref{KhT-1} and Corollary \ref{KhT-2} are usually known as Khovanskii--Teissier inequalities. For $m=2$, inequality (\ref{divalcor1}) provides the Alexandrov--Fenchel type inequality  (\ref{divAF}) for nef  divisors. For $m=n$,  inequality (\ref{divalcor1}) is symmetric in the nef  divisors   $ D_1,\dots, D_n$.

\begin{corollary}  For any $2\leq m\leq n$ and for any collection of nef   divisors  $D_1,  D_2$, $ D_{m+1},\dots, D_n$, we have:
\begin{equation}\label{nefcor2}\begin{aligned}
 [ D_1,\dots, D_1, D_{m+1},\dots, D_n]^ {\frac{1}{m}}+  [ D_2,\dots, D_2, D_{m+1},\dots, D_n]^ {\frac{1}{m}}\\
\leq  [ D_1 + D_2,\dots, D_1 + D_2, D_{m+1},\dots, D_n]^{\frac{1}{m}}.
\end{aligned}\end{equation}
\end{corollary}

For $m=n$, inequality (\ref{nefcor2})
provides the  Brunn--Minkowski  type inequality for nef  divisors.

\section{Algebraic proofs of geometric inequalities}\label{sect8}

The Alexandrov--Fenchel inequality  for convex bodies in $\Bbb R^n$  automatically follows from the corresponding algebraic inequality (see  \cite{Kh88}).

\begin{lemma}\label{intBKK}  For any $n$-tuple of integral polytopes $\Delta_1,\dots, \Delta_n\subset \Bbb R^n$, the Alexandrov--Fenchel inequality (\ref{eq4}) holds.
\end{lemma}

\begin{proof} Let $A_i$ be the intersection of $\Delta_i$ with the lattice $ \Bbb Z^n$ and let $L_{A_i}$ be the spaces corresponding to the sets $A_i$. By the BKK theorem (see Theorem \ref{BKK}), the intersection index $[L_{A_1},\dots,L_{A_n}]$ is equal to  $n!MV_n(\Delta_1,\dots, \Delta_n)$. Thus, the lemma follows from Corollary \ref{algAF}.
\end{proof}

\begin{theorem} The Alexandrov--Fenchel inequality (\ref{eq4}) holds for any $n$-tuple of convex bodies in $\Bbb R^n$.
\end{theorem}

\begin{proof} By Lemma \ref{intBKK}, the  Alexandrov--Fenchel inequality holds for integral polytopes. Since volume is a homogeneous polynomial on the cone of convex bodies, the inequality also  holds for polytopes with rational vertices. Any convex body can be approximated in the Hausdorff metric by polytopes with rational vertices.
This implies the theorem because mixed volume is continuous in the Hausdorff metric.
\end{proof}

\section{Minkowski-type criterion for divisors}\label{sect9}

Minkowski discovered (see Theorem \ref{Mink})  a  necessary and sufficient condition on a $k$-tuple (where  $k\leq n)$)  of  convex bodies in $\Bbb R^n$, under  which the mixed volume of any $n$-tuple of convex bodies in $\Bbb R^n$, containing  the given $k$-tuple of convex bodies, vanish. In this section, we present an analogous condition  on $k\leq n$ basepoint free linear systems in the $n$-dimensional projective variety under which generic members of these systems have empty intersection.

Let $X$ be an $n$-dimensional projective variety.
With any basepoint free divisor $D$ on $X$ one associates a linear system of effective divisors $D+(f)$, where $f\in L(D)$ linearly equivalent to $D$.

\begin{definition} The {\it  Kodaira dimension} of a basepoint free divisor $D$ is the dimension of the image $Y=\Phi_{L(D)}(X)$ of $X$ under the Kodaira map $\Phi_{L(D)}:X\to \Bbb P (L^*(D))$.
\end{definition}

One can show that {\it the Kodaira dimension of a divisor  $D$} is equal to the real dimension of the NO body $\Delta(L(D))$ of the space $L(D)$ (see  \cite{K-K12}).

\begin{definition} A set $\{D_1,\dots,D_k\}$ of $k\leq n $
basepoint free  divisors in a projective variety $X$ is {\it  dependent} if   for some nonempty subset $J\in \{1,\dots,k\}$  the divisor  $D_J=\sum_{j\in J} D_j$ has Kodaira dimension smaller  than the number $|J|$ of elements in the set $J$.
\end{definition}

\begin{theorem}[Minkowski type criterion]\label{theorem9.1}
 Let $D_1,\dots,D_k$ be a collection of $k\leq n$ basepoint free  divisors on an irreducible $n$-dimensional projective variety~$X$. For generic functions $f_i\in L(D_i)$, the effective divisors $D_1+(f_1),\dots, D_k+(f_k)$ have  empty intersection if and only if the divisors $D_1,\dots,D_k$ are dependent. In other words, they have empty intersection if and only if for some set $J$ the real dimension  of the NO body $\Delta( L(D_J))$ is smaller than $|J|$.
\end{theorem}

A proof of Theorem \ref{theorem9.1} can be found in  \cite{K-K16}.

\begin{corollary} Let $D_1,\dots,D_n$ be an $n$-tuple of basepoint free  divisors on~$X$. Then the intersection index $[D_1,\dots,D_n]$ is equal to zero if and only if the divisors $D_1,\dots,D_n$ are dependent. In the other words, $[D_1,\dots,D_n]=0$ if and only if for some $J\in \{1,\dots,n\}$ the real dimension of the NO body $\Delta(L(D_j))$ is smaller than $|J|$.

\end{corollary}





\begin{thebibliography}{widestlabel}

 \bibitem{Isk03} V. A. Iskovskikh, $b$-divisors and Shokurov functional algebras (Russian). \emph{Tr. Mat. Inst.
Steklova}  \textbf{240} (2003), Biratsion. Geom. Linein. Sist. Konechno Porozhdennye Algebry, 8--20;
translation in \emph{Proc. Steklov Inst. Math.}  \textbf{240} (2003), no.~1, 4--15.

 \bibitem{K-K08} K. Kaveh, and A. G. Khovanskii, Convex bodies and algebraic equations
on affine varieties, 2008,  	 https://arxiv.org/pdf/0804.4095.

 \bibitem{K-K10} K. Kaveh, and A. G. Khovanskii, Mixed volume and an extension of intersection theory of divisors.
 \emph{Mosc. Math. J.}  \textbf{10} (2010), no.~2, 343--375,~479.

 \bibitem{K-K12} K. Kaveh, and A. G. Khovanskii, Newton--Okounkov bodies, semigroups of integral points, graded algebras and intersection theory.  \emph{Ann. of Math.}  \textbf{176} (2012), no.~2, 925--978.

 \bibitem{K-K14} K. Kaveh, and A. Khovanskii,  Note on the Grothendieck group of subspaces of rational functions and Shokurov's $b$-divisors.  \emph{Canadian Mathematical Bulletin},  \textbf{7} (2014), no~3, 562--572.

 \bibitem{K-K16} K. Kaveh,  A. Khovanskii,  Complete intersections in spherical varieties.  \emph{Selecta Mathematica.}  \textbf{22} (2016), no~4, Special Issue: The Mathematics of Joseph Bernstein, 2099--2141.

 \bibitem{Kh88} A. Khovanskii, \emph{Algebra and mixed volumes}. Appendix in:  Y.D.~Burago and V.A.~Zalgaller,  \emph{Geometric inequalities}, Springer-Verlag, Berlin and New York. \textbf{285} (1988), 182--207, translated from
the Russian by A. B. Sosinskii, Springer Series in Soviet Mathematics.


 \bibitem{Kh92} A. G. Khovanskii, The Newton polytope, the Hilbert polynomial and sums of finite
sets,  \emph{Funktsional. Anal. i Prilozhen}.  \textbf{26} (1992), 57--63, 96, translated in
\emph{Funct. Anal. Appl.}  \textbf{26} (1992), 276--281. MR 1209944.

 \bibitem{Kh25} A. Khovanskii, Polynomial functions on semigroups, \emph{Paper in preparation}.

 \bibitem{Kh-T14} A. Khovanskii, and V. Timorin, On the theory of coconvex bodies.
 \emph{Discrete \& Computational Geometry}.  \textbf{52} (2014), no~4, 806--823.

 \bibitem{LM09} R. Lazarsfeld, and M. Mustata, Convex bodies associated to linear
series,  \emph{Ann. Sci.  \'Ec. Norm. Sup\'er.}  \textbf{42} (2009), 783--835.

 \bibitem{Tei79} B. Teissier,  Du th\'eor\`eme de l'index de Hodge aux in\'egalit\'es
isop\'erim\'etriques   \emph{C.~R.~Acad. Sci. Paris S\'er. A-B}.  \textbf{288} (1979), no~4, A287--A289.

 \bibitem{ZarS} O. Zariski, and P. Samuel, \emph{Commutative Algebra. Vol. II}, The University
Series in Higher Mathematics, D. Van Nostrand Co., Princeton, NJ, (1960).


\end{thebibliography}
\end{document}